\documentclass[a4paper,12pt]{preprint}
\usepackage[margin=1.3in]{geometry}  
\usepackage[full]{textcomp}

\usepackage{mathalfa}
\usepackage{microtype}

\usepackage{enumitem}
\usepackage{amsmath}
\usepackage{amsfonts} 
\usepackage{amssymb}             
\usepackage{comment}
\usepackage{breakurl}
\usepackage{mathrsfs}
\usepackage{booktabs}
\usepackage{mathtools}

\usepackage{tikz}
\usepackage{amsthm}                
\usepackage{mhequ}
\usepackage{hyperref}

\usepackage{accents}

\hypersetup{pdfstartview=}

\addtocontents{toc}{\protect\hypersetup{hidelinks}}

\DeclareSymbolFont{sfoperators}{OT1}{ptm}{m}{n}
\DeclareSymbolFontAlphabet{\mathsf}{sfoperators}

\makeatletter
\def\operator@font{\mathgroup\symsfoperators}
\makeatother

\numberwithin{equation}{section}

\newtheorem{thm}{Theorem}[section]

\newtheorem{lem}[thm]{Lemma}
\newtheorem{prop}[thm]{Proposition}

\newtheorem{cor}[thm]{Corollary}

\newtheorem{assumption}[thm]{Assumption}

\theoremstyle{remark}
\newtheorem{remark}[thm]{Remark}
\newtheorem{rmk}[thm]{Remark}

\makeatletter
\def\th@newremark{\th@remark\thm@headfont{\bfseries}}

\def\bdiamond{\mathop{\mathpalette\bdi@mond\relax}}
\newcommand\bdi@mond[2]{%
	\vcenter{\hbox{\m@th
			\scalebox{\ifx#1\displaystyle 2.6\else1.8\fi}{$#1\diamond$}%
	}}%
}

\def\bDiamond{\mathop{\mathpalette\bDi@mond\relax}}
\newcommand\bDi@mond[2]{%
	\vcenter{\hbox{\m@th
			\scalebox{\ifx#1\displaystyle 2.6\else1.2\fi}{$#1\Diamond$}%
	}}%
}
\makeatletter

\definecolor{darkgreen}{rgb}{0.1,0.7,0.1}
\definecolor{darkred}{rgb}{0.7,0.1,0.1}
\definecolor{darkblue}{rgb}{0,0,0.7}
\newcommand{\PP}{\mathbb{P}}
\newcommand{\RR}{\mathbb{R}}

\newcommand{\dus}{\frac{\Delta_{j(s)}U}{t_{j(s)+1}-t_{j(s)}}}

\newcommand{\unm}{u^{(n)}_{m}}
\newcommand{\tw}{\widetilde{w}}

\newcommand{\un}{u^{(n)}}

\newcommand{\cC}{\mathcal{C}}

\newcommand{\fF}{\mathcal{F}}

\newcommand{\tv}{\widetilde{v}}
\newcommand{\iI}{\mathcal{I}}

\newcommand{\dws}{\frac{\Delta_{j(s)}W}{t_{j(s)+1}-t_{j(s)}}}
\newcommand{\nN}{\mathcal{N}}

\newcommand{\sS}{\mathcal{S}}

\newcommand{\xX}{\mathcal{X}}
\newcommand{\yY}{\mathcal{Y}}

\newcommand{\E}{\mathbf{E}}
\newcommand{\N}{\mathbf{N}}
\newcommand{\R}{\mathbb{R}}
\newcommand{\V}{\mathbf{V}}

\newcommand{\sn}{\widetilde{N}\big( v; s \big) }

\newcommand{\1}{\mathbf{1}}

\newcommand{\eps}{\varepsilon}

\newcommand{\ssm}{\mathcal{S}_{mar}}
\newcommand{\ssq}{\mathcal{S}_{qua}}

\newcommand{\sm}{\mathcal{S}^*_{\text{mar}}}
\newcommand{\sq}{\mathcal{S}^*_{\text{qua}}}

\makeatletter 
\newcommand{\cov}{{\operator@font cov}}
\newcommand{\var}{{\operator@font var}}
\newcommand{\corr}{{\operator@font corr}}
\newcommand{\diam}{{\operator@font diam}}
\newcommand{\Av}{{\operator@font Av}}
\newcommand{\trig}{{\operator@font trig}}
\newcommand{\Enh}{{\operator@font Enh}}
\makeatother

\newcommand{\ini}{\text{ini}}
\newcommand{\mar}{\text{mar}}
\newcommand{\noi}{\text{noi}}
\newcommand{\qua}{\text{qua}}

\colorlet{symbols}{blue!90!black}
\colorlet{testcolor}{green!60!black}

\def\${|\!|\!|}

\makeatletter
\def\DeclareSymbol#1#2#3{\expandafter\gdef\csname MH@symb@#1\endcsname{\tikz[baseline=#2,scale=0.15,draw=symbols]{#3}}\expandafter\gdef\csname MH@symb@#1s\endcsname{\scalebox{0.7}{\tikz[baseline=#2,scale=0.15,draw=symbols]{#3}}}}
\def\<#1>{\csname MH@symb@#1\endcsname}
\makeatother

\setlist[itemize]{topsep=3pt,itemsep=1.5pt,parsep=0pt}

\def\cent#1{\mathopen{{\langle\kern-0.3em\rangle}}#1\mathclose{{\langle\kern-0.3em\rangle}}}

\def\d{\partial}

\begin{document}

\title{A Wong-Zakai theorem for the stochastic mass-critical NLS}
\author{Chenjie Fan$^1$ and Weijun Xu$^2$}
\institute{Academy of Mathematics and Systems Science and Hua Loo-Keng Key Laboratory of Mathematics, Chinese Academy of Sciences, Beijing, 100190, China, \email{cjfanpku@gmail.com}
\and Beijing International Center for Mathematical Research, Peking University, China, \email{weijunxu@bicmr.pku.edu.cn}}

\maketitle

\begin{abstract}
We prove a Wong-Zakai theorem for the defocusing stochastic mass-critical nonlinear Schr\"odinger equation (NLS) on $\R$. The main ingredient are careful mixtures of bootstrapping arguments at both deterministic and stochastic levels. Several subtleties arising from the proof mark the difference between the dispersive case and corresponding situations in SDEs and parabolic stochastic PDEs, as well as the difference between the large-$n$ case and the limiting ($n=\infty$) case. 
\end{abstract}

\setcounter{tocdepth}{2}
\microtypesetup{protrusion=false}
\tableofcontents
\microtypesetup{protrusion=true}

\section{Introduction}

\subsection{Statement of the result}
We continue our study of the defocusing mass-critical stochastic nonlinear Schr\"odinger equation on $\R$ with a conservative noise. Consider the model
\begin{equation} \label{eq:snls_crit}
i \d_t u + \Delta u = |u|^{4} u + u \circ \dot{W}\;, \quad u(0,\cdot) = f \in L_{\omega}^{\infty}L_{x}^{2} := L^{\infty} \big(\Omega,L_x^2(\R) \big). 
\end{equation}
Here, $W$ is a Wiener process on real-valued functions with proper integrability conditions, and $\circ$ denotes the Stratonovich product. This choice of noise and product assures pathwise mass conservation.

\begin{assumption} \label{as:noise_initial}
	Throughout, we assume the initial data $f \in L_{\omega}^{\infty}L_{x}^{2}$, and there exists $\Lambda_{\ini} > 0$ such that
	\begin{equation}
	\|f\|_{L_{\omega}^{\infty}L_{x}^{2}} \leq \Lambda_{\ini} \lesssim 1. 
	\end{equation}
	The Wiener process $W(t,\cdot)$ has the form
	\begin{equation}
	W(t,x) := \sum_{k \in \N} B_{k}(t) V_{k}(x), 
	\end{equation}
	where $\{B_k\}$ are independent standard Brownian motions on the probability space $(\Omega, \fF, \PP)$, and $\{V_k\}$ is a sequence of real-valued functions on $\R$. Furthermore, they satisfy
	\begin{equation}\label{eq:summable}
	\sum_{k \in \N} \big( \|V_k\|_{L_x^1} + \|V_k\|_{L_x^{\infty}} \big) =: \sum_{k \in \N} \Lambda_k \leq \Lambda_{\noi}\lesssim 1
	\end{equation}
	for some $\Lambda_{\noi} > 0$. 
\end{assumption}

\begin{rmk}
From the view point of the assumption \eqref{eq:summable} and the fact that the topic of  this article is of well-posedness nature, the noise we considered is not very different from the special form $W(x,t)=V(x)B(t)$ for $V\in L_{x}^{1}\cap L_{x}^{\infty}$. However, non-degenerate noise does arise in some other situations, so we emphasize that our results also work for infinite dimensional noise. It should be remarked we do not require any smoothness of the noise in the space variable. 

It may also be of interest to extend our results to more general noise based on the language of $\gamma$-Radonifying operators. We do not handle this technical issue here, though we believe most of our arguments can also work in that situation. Also, this paper does not aim to handle very singular noise (such as space-time white noise) that requires renormalisation. 
\end{rmk}

Equation \eqref{eq:snls_crit} is called mass-critical since its deterministic version (in general dimension $d$) has the following scaling property: suppose $v$ satisfies
\begin{equation} \label{eq:nls_crit}
i \d_t v + \Delta v = |v|^{\frac{4}{d}}v,\quad v(0)=v_{0}\in L_{x}^{2}(\R^{d}),
\end{equation}
then for every $\lambda>0$, the rescaled function $v_{\lambda}(t,x):= \lambda^{-\frac{d}{2}} v(t/\lambda^2, x/\lambda)$ satisfies the same equation with initial data $\lambda^{-\frac{d}{2}} v_0(\cdot/\lambda)$, whose $L_{x}^2$ norm is invariant under the scaling. It is called defocusing since the associated Hamiltonian is coercive. 

The well-posedness for \eqref{eq:nls_crit} is highly non-trivial for general $L_{x}^{2}$ initial data. It was proved by Dodson (\cite{dodson2012global, dodson2467global, dodson2016global}) that for every initial data in $L_{x}^2(\R^d)$, $v$ has a global space-time bound in a suitable Strichartz space. As a consequence, $v$ scatters. 

These results are crucial if one wants to construct global solutions to the stochastic equation. In \cite{snls_mass_critical, snls_subcritical_approx, snls_critical_Zhang}, a global solution to \eqref{eq:snls_crit} was constructed. The purpose of this article is to present a Wong-Zakai type result for \eqref{eq:snls_crit}. 

Let $\pi^{(n)}$ be a sequence of partitions of $[0,1]$ of the form
\begin{equation}
\pi^{(n)} := \big\{0 = t_{0}^{(n)} <t_{1}^{(n)} <\cdots <t_{n}^{(n)} = 1\big\} 
\end{equation}
such that $\|\pi^{(n)}\| := \max_{j} \big| t_{j+1}^{(n)} - t_{j}^{(n)}\big| \rightarrow 0$ as $n \rightarrow +\infty$. For every $n$, let $W^{(n)}$ be such that $W^{(n)}(t^{(n)}_j) = W(t_j^{(n)})$ for every $t_{j}^{(n)} \in \pi^{(n)}$, and linearly interpolates in between. Let $u^{(n)}$ be the solution to
\begin{equation} \label{eq:un}
i \d_t u^{(n)} + \Delta u^{(n)} = |u^{(n)}|^{4} u^{(n)} + u^{(n)} \frac{{\rm d} W^{(n)}}{{\rm d} t}\;, \quad u^{(n)}(0,\cdot) = u(0,\cdot)\;.  
\end{equation}
Unlike the SPDE \eqref{eq:snls_crit} which needs to be formulated via It\^o's stochastic integration, \eqref{eq:un} is well defined in the classical sense since $W^{(n)}$ is piecewise linear. Thus, \eqref{eq:un} is simply a classical PDE (with randomness), and we can rewrite as
\begin{equation}\label{eq: un2}
i\d_t u^{(n)} + \Delta u^{(n)} = |u^{(n)}|^{4} u^{(n)} + u^{(n)} \cdot \frac{W(t_{j+1}^{(n)})-W(t_{j}^{(n)})}{t_{j+1}^{(n)}-t_{j}^{(n)}}\;, \quad t\in [t_{j}^{(n)},t_{j+1}^{(n)}]\;.
\end{equation}
Given the results of Dodson (\cite{dodson2012global},\cite{dodson2467global},\cite{dodson2016global}), for every realization of $W$, one can see that $u^{(n)}$ is globally well-posed. Such global existence of $u^{(n)}$ can be derived from a purely perturbative viewpoint (see Remark~\ref{rem: useful}). Our main result is the following. 

\begin{thm} \label{th:main}
	 Let $u$ be the solution to \eqref{eq:snls_crit} as constructed in \cite{snls_mass_critical, snls_subcritical_approx}. Let $\un$ be the solution to \eqref{eq:un} with the same initial data $u(0,\cdot) \in L_{\omega}^{\infty}\big(\Omega, L_{x}^{2}(\R)\big)$. Then for every $\rho \geq 1$, we have
	\begin{equation}\label{eq: mainconv}
	\E \|u^{(n)} - u\|_{\xX(0,1)}^{\rho} \rightarrow 0
	\end{equation}
	as $n \rightarrow +\infty$, where we used the shorthand notation
	\begin{equation} \label{eq:norms}
	\xX_{1}(\iI) = L_{t}^{\infty} \big(\iI, L_x^2(\R)\big), \; \xX_2(\iI) = L_{t}^{5}\big(\iI, L_{x}^{10}(\R)\big), \; \xX(\iI) = \xX_1(\iI) \cap \xX_2(\iI).
	\end{equation}
\end{thm}

\begin{remark}
	Even to ensure the convergence \eqref{eq: mainconv} for a single fixed $\rho \geq 1$, it is crucial to assume our initial data to be in $L_{\omega}^{\infty}L_{x}^{2}$ rather than $L_{\omega}^{\rho'}L_{x}^{2}$ no matter how large $\rho'$ is. Also, the convergence \eqref{eq: mainconv} works for any finite $\rho$ but not for $\rho = \infty$. 
\end{remark}

\begin{remark}
The above result can be extended to sub-critical models.
\end{remark}

\begin{remark}
Due to the mass-critical nature of the problem and the fact we are working with general $L_{x}^{2}$ data, the dynamics in the time interval $[0,1]$ should not be understood as a short time dynamics, but rather a truly nonlinear one (instead of being a perturbation of some linear dynamics). In fact, the exactly same argument works for $[0,T]$ for any $T>0$. On the other hand, our result does not give information on the asymptotic behaviour of the solution. One needs strong decay estimates for the solutions in order to have a solution theory on the time interval of the whole $\RR$. It is still out of reach at this moment, but see \cite{stochastic_linear_schrodinger_decay} for the linear case with small noise and \cite{snls_scattering} for the critical nonlinear case with a noise with a time-decay noise. 
\end{remark}

\begin{remark} \label{rm:canonical}
One can also understand this result as an alternative construction of the solution to \eqref{eq:snls_crit}. The convergence \eqref{eq: mainconv} has two levels of meanings. First, it says $u^{(n)}$ converges in $L_{\omega}^{\rho}\xX(0,1)$ and the limit solves \eqref{eq:snls_crit}. Second, the limiting process coincides with the process $u$ constructed in \cite{snls_mass_critical}. The second point is meaningful for the following reason: even though the construction in \cite{snls_mass_critical} is canonical in the sense that it is the unique limit under a natural class of approximations (see Theorem \ref{th:snls_mass_critical} for more details), the uniqueness is not in the sense of a prescribed function space. Hence, Theorem~\ref{th:main} could also serve as a demonstration that the solution constructed in \cite{snls_mass_critical} is the right one to consider. 
\end{remark}

\begin{remark}
Pathwise mass conservation is not strictly necessary as long as one can control the growth of $L_x^2$-norm of the solution independent of the realization of the path. Nevertheless, from the viewpoint of dispersive PDEs, we still restrict ourselves to the model where pathwise mass conservation law holds for the future study of long time behavior. 

On the other hand, such (deterministic) a priori control does play an essential role in this article (especially for $L_{\omega}^{\rho}$ convergence). One may also want to state and prove similar results for energy (sub-)critical problems with $H_x^1$ initial data, but in that case one may only get convergence in probability instead of $L_{\omega}^{\rho}$ due to the lack of such control for energy growth. 
\end{remark}

\subsection{Review of the construction of solution to (\ref{eq:snls_crit})}

The It\^o form of \eqref{eq:snls_crit} is formally given by
\begin{equation}
i \d_t u + \Delta u = |u|^{4} u + u \dot{W} - \frac{i}{2} \V^2 u, 
\end{equation}
where the product between $u$ and $\dot{W}$ is in the It\^o sense, and $\V^2 = \sum_{k} V_{k}^{2}$ is the It\^o-Stratonovich correction. 

Let $\phi \in \cC_{c}^{\infty}(\R^{+}, \R)$ such that $\phi \in [0,1]$, $\phi = 1$ on $[0,1]$ and vanishes outside the interval $[0,2]$. For every $m>0$, let $\phi_m(\cdot) := \phi(\cdot/m)$. We summarize the construction in \cite{snls_mass_critical} in the following. 

\begin{thm} [\cite{snls_mass_critical}] 
	\label{th:snls_mass_critical}
	Fix $f \in L_{\omega}^{\infty}L_{x}^{2}$. For every $\rho \geq 1$ and $m>0$, there exists a unique $u_m \in L_{\omega}^{\rho}\xX(0,1)$ adapted to the filtration generated by $W$ such that
	\begin{equation} \label{eq:um_Duhamel}
	\begin{split}
	u_m(t) = &e^{it\Delta}f - i \int_{0}^{t} e^{i(t-s)\Delta} \Big( \phi_m\big( \|u_m\|_{\xX_2(0,t)} \big) |u_m(s)|^{4} u_m(s) \Big) {\rm d}s\\
	&- i \int_{0}^{t} e^{i(t-s)\Delta} \big( u_{m}(s) {\rm d} W(s) \big)- \frac{1}{2} \int_{0}^{t} e^{i(t-s)\Delta} \big( \V^2 u_m(s) \big) {\rm d}s, 
	\end{split}
	\end{equation}
	where the integral with respect to $W$ is in the It\^o sense, and the formula holds in $L_{\omega}^{\rho}\xX(0,1)$. 
	
	As a consequence, we deduce that the same $u_m$ belongs to $L_{\omega}^{\rho}\xX(0,1)$ for every $\rho \geq 1$. Furthermore, $\forall \rho \geq 1$, $\{u_m\}_{m}$ is Cauchy in $L_{\omega}^{\rho}\xX(0,1)$ as $m \rightarrow +\infty$, and the limit $u$ satisfies the Duhamel formula
	\begin{equation} \label{eq:u_Duhamel}
	\begin{split}
	u(t) = &e^{it\Delta}f - i \int_{0}^{t} e^{i(t-s)\Delta} \Big( |u(s)|^{4} u(s) \Big) {\rm d}s\\
	&- i \int_{0}^{t} e^{i(t-s)\Delta} \big( u(s) {\rm d} W(s) \big) - \frac{1}{2} \int_{0}^{t} e^{i(t-s)\Delta} \big( \V^2 u(s) \big) {\rm d}s, 
	\end{split}
	\end{equation}
	where the stochastic integral is again in the It\^o's sense, and the two sides are equal in $L_{\omega}^{\rho}\xX(0,1)$. 
\end{thm}

\begin{rmk}
In \cite{snls_mass_critical}, the results were stated for spatially smooth noise, but they are valid with essentially same proof only with Assumption~\ref{as:noise_initial} in this article. Furthermore, the current article considers a large class of discretisations of the model \eqref{eq:snls_crit} with stronger estimates, and hence one can also use the arguments here to recover the results in \cite{snls_mass_critical} with the noise in Assumption~\ref{as:noise_initial}. 
\end{rmk}

\begin{rmk} \label{rm: umuniqueness}

Note that the construction of solution $u_m$ to \eqref{eq:um_Duhamel} does not follow from a Picard iteration regime, but rather via sub-critical approximations. But we do have natural uniqueness for \eqref{eq:um_Duhamel} in the sense that there is only one element in the prescribed space which satisfies the formula \eqref{eq:um_Duhamel}. 

On the other hand, as already mentioned in Remark~\ref{rm:canonical}, the uniqueness of $u$ is not in the sense of prescribed function space but rather as the unique limit of natural approximations. But this solution should not be understood as a weak solution derived from compactness arguments, which are usually not unique for a given initial data. We finally remark that it is natural to expect $u=u_{m}$ for $t \leq \tau_{m}$, where $\tau_{m}$ is the stopping time when the $\chi_{2}$-norm of $u$ hits $m$. We do not have a proof for this. If this is true, then we will also have uniqueness for the original equation \eqref{eq:snls_crit} in the sense of prescribed function space. 
\end{rmk}

\subsection{Background}

\subsubsection{Mass-critical NLS}

Mass-critical NLS is a typical model for nonlinear dispersive equations. The local well-posedness is well known, and can be established via a Picard iteration scheme with Strichartz estimates. One may refer to \cite{cazenave1989some} as well as the textbooks \cite{cazenave2003semilinear} and \cite{tao2006nonlinear}. Global well-posedness for the (defocusing) mass-critical NLS with general $L_{x}^{2}$ initial data has been a famous open problem, and was finally solved by Dodson in \cite{dodson2012global, dodson2467global, dodson2016global} (see also the reference therein for more background). We summarize Dodson's result for the 1D model below. 

\begin{thm} [Dodson]
	\label{th:Dodson} Let $v$ solve the equation
	\begin{equation} \label{eq:nls_crit_app}
	i \d_t v + \Delta v = |v|^{4} v\;, \quad v(0,\cdot)=v_{0}\in L^2(\R). 
	\end{equation}
	Then $v$ is global and
	\begin{equation} \label{eq: sc}
	\|v\|_{\xX_2(\R)} \lesssim_{\|v_{0}\|_{L_{x}^{2}}} 1, 
	\end{equation}
	where we recall from \eqref{eq:norms} that $\xX_2(\iI) = L_{t}^{5} \big( \iI, L_{x}^{10}(\RR) \big)$. As a consequence, there exists $v_{+}\in L_{x}^{2}$ depending on $v_0$ such that
	\begin{equation}
	\|v(t)-e^{it\Delta}v_{+}\|_{L_{x}^{2}} \rightarrow 0
	\end{equation}
    as $t \rightarrow \infty$. 
\end{thm}

\subsubsection{Well-posedness of stochastic NLS}

We mainly focus on the stochastic NLS with a multiplicative noise on the whole space. The study of stochastic NLS with a multiplicative noise was initiated in the work of de Bouard and Debussche (\cite{dBD}, \cite{Debussche_H1}), where local and global well-posedness for sub-critical non-linearities was established. See also refinements and further developments in \cite{ Rockner1, Rockner2, hornung2016nonlinear}. In \cite{snls_mass_critical, snls_subcritical_approx}, we extended the result of \cite{dBD}, and proved global well-posedness for 1D mass-critical model with $L_{x}^{2}$ initial data. Later, Zhang (\cite{snls_critical_Zhang}) generalised the results for the mass-critical model to all dimensions and also proved well-posedness for the energy-critical model via a different method (rescaling method) and with different assumptions on the noise as well as different notions of solution. 

Note that if one is interested in $L_{\omega}^{\rho}$ bounds as in the works \cite{dBD, snls_mass_critical, snls_subcritical_approx}, it should be expected mass-critical model is very different from energy-critical models. Indeed, such bounds are not yet available for the latter. 

We conclude with a short discussion on the difference between (local) well-posedness for deterministic and stochastic NLS. Note that such difference will arise even for very simple noise $W(x,t)=V(x)B(t)$ and sub-critical nonlinearity.  

Following the local well-posedness of deterministic NLS, one may want to use Duhamel formula \eqref{eq:u_Duhamel} to construct solutions via a Picard iteration scheme in space of the type $L^{\rho}_{\omega} \yY$ for some (space-time) metric space $Y$. However, such a procedure fails for the following reason. If $u\in L^{\rho}_{\omega} \yY$, then the nonlinear term (which is $|u|^{4}u$ in our case) is expected to live in $L^{\rho/5}_{\omega} \yY'$, no matter which spaces $(Y,Y')$ one chooses. To overcome this difficulty, one needs to explore the so-called pathwise mass conservation law in the model. Thus, even in the local theory of \cite{dBD, snls_mass_critical}, some non-perturbative information is used.

Finally, we remark that there are different notions of local solutions in stochastic NLS. Some are easier to construct, but may have more difficulty to be extended globally. On the other hand, if one uses the notions of solutions as in \cite{dBD, snls_mass_critical}, then those local solutions are very easy to be extended to be global, but indeed long time dynamics do appear even in very short time (though with small probability). It is natural to expect these solutions should agree with each other. One motivation of this article is that, by showing a Wong-Zakai convergence, we demonstrate the solution constructed in our previous work \cite{snls_mass_critical} is natural in the sense it can be approximated by classical solutions of PDEs.

\subsubsection{Wong-Zakai convergence}

The classical Wong-Zakai type theorem refers to a series of pioneering results by Wong and Zakai (\cite{WZ_1, WZ_2}) on one dimensional SDEs and by Stroock and Varadhan (\cite{SV_support}) on multidimensional SDEs, which roughly assert that if $B^{(n)}$ converges to a (finite dimensional) Brownian motion $B$, then the solution $X^{(n)}$ to the multidimensional SDE
\begin{equation}
{\rm d} X^{(n)}(t) = \mu\big(X^{(n)}(t)\big) {\rm d}t + \sigma \big(X^{(n)}(t)\big) {\rm d} B^{(n)}(t) 
\end{equation}
converges to the solution $X$ to the \textit{Stratonovich SDE}
\begin{equation}
{\rm d} X(t) = \mu\big(X(t)\big) {\rm d}t + \sigma\big(X(t)\big) \circ {\rm d}B(t)\;. 
\end{equation}
The convergence statement as well as the rate (in terms of $n$) is directly related to sample path continuity of $X_t$. 

As for analogous questions for parabolic stochastic PDEs, it is natural to consider the model
\begin{equation} \label{eq:spde_parabolic}
{\rm d} u(t) = \Delta u {\rm d}t + f(u) {\rm d}t + g(u) \circ {\rm d} W(t)
\end{equation}
for some Wiener process $W$. The linear operator $e^{t\Delta}$ has a strong smoothing effect (it immediately turns any initial data into a smooth function). Hence, as long as the noise is not too singular  and the nonlinear effect is not too strong, the dissipative system essentially behaves like high dimensional ODEs, and it is reasonable to expect the Wong-Zakai approximations to converge.

On the other hand, if the noise is singular (for example, $W$ being cylindrical Wiener process on $L_x^{2}$), then the problem becomes much subtler as the singularity of the noise is strong enough so that the Stratonovich formulation does not exist. The question for singular parabolic SPDEs has been open for a long time, and was  successfully handled in \cite{Hairer_WZ} with the framework of regularity structures (with proper renormalizations). 

The aim of this article is to prove a Wong-Zakai type theorem for the nonlinear dispersive PDEs \eqref{eq:snls_crit} (in contrast to dissipative PDEs). In some sense, dispersive equations behave less like high dimensional ODEs than dissipative ones since the linear propagator $e^{it\Delta}$ does not have smoothing effects on $L_{x}^2$ initial data. In particular, it does not smooth out high frequencies of the solution flow, especially when one has a critical nonlinearity. Indeed, even the solution to the linear deterministic equation does not have any H\"older continuity as a flow in $L_{x}^2$. Also, we point out here the nonlinearity we are considering are $L_{x}^2$-critical, so it is strong enough that all levels of frequencies should be taken into account. On the other hand, the noise we consider here, though not smooth, is essentially finite dimensional and does not require subtle renormalization treatments as in singular parabolic PDEs.

Since typical Wong-Zakai convergences are intimately related to the time regularity of the solution flow, and that even the linear Schr\"odinger flow is not H\"older continuous in $L_x^2$, it is a priori unclear whether Wong-Zakai convergence is true even if the limiting equation is well defined. Hence, it is our interest to show that this is indeed true, though our convergence statement does not have a rate.

\subsection{Sketch of proof of the main theorem}

In order to show the convergence of $u^{(n)}$ to $u$ , we introduce the intermediate process $\un_{m}$ which solves the equation
\begin{equation} \label{eq:unm}
i \d_t u_{m}^{(n)} + \Delta u_m^{(n)} = \phi_{m} \big( \|u_m^{(n)}\|_{\xX_2(0,t)} \big) |u_m^{(n)}|^{4} u_m^{(n)} + \unm\frac{{\rm d} W^{(n)}}{{\rm d}s}, \quad u_m^{(n)}(0,\cdot) = f. 
\end{equation}
We first make a simple observation. Unlike $u_m$ and $u$ (in \eqref{eq:um_Duhamel} and \eqref{eq:u_Duhamel}) whose formulations rely on stochastic integral, the processes $\un$ and $\unm$ are pathwisely defined for every finite $n$, and one easily verifies $\un=\unm$ on $[0,\tau_{m}^{(n)}]$, where $\tau_{m}^{(n)}$ is such that $\|u^{(n)}\|_{\xX_2([0,\tau_m^{(n)}])} = m$. 

The main ingredient to prove Theorem \ref{th:main} is the following uniform bound on $\un$ and $\unm$. 

\begin{thm} \label{th:unm_uniform_bd_stable}
	Given Assumption~\ref{as:noise_initial} for the noise and initial data, let $\un$ and $\un_{m}$ solve \eqref{eq:un} and \eqref{eq:unm}. Then for every $\rho \geq 1$, we have
	\begin{equation} \label{eq:unm_uniform_bd}
	\|u^{(n)}_{m}\|_{L_{\omega}^{\rho}\xX(0,1)} \lesssim_{\Lambda_{ini}, \Lambda_{noi},\rho} 1\;, \quad \|u^{(n)}\|_{L_{\omega}^{\rho}\xX(0,1)} \lesssim_{\Lambda_{ini}, \Lambda_{noi},\rho} 1 
	\end{equation}
	uniformly	in $n$ and $m$. 
\end{thm}

Most of the estimates in this article involve constants depending $\Lambda_{ini}$ and $\Lambda_{noi}$. Since they are fixed, we do not emphasize these dependences in the rest of the article. 

The uniform bound \eqref{eq:unm_uniform_bd} implies the following stability result, which will allow us to reduce the problem to the one with regular noise and initial data. 

\begin{cor}\label{cor: stab}
Let $\widetilde{u}^{(n)}_{m}$ and $\widetilde{u}_m$ denote the solutions to \eqref{eq:unm} and \eqref{eq:um_Duhamel} with initial data $\widetilde{f} \in L_{\omega}^{\infty}L_{x}^{2}$ and noise
	\begin{equation}
	\widetilde{W}(t,x) = \sum_{k \in \N} B_{k}(t) \widetilde{V}_{k}(x)
	\end{equation}
	satisfying also Assumption~\ref{as:noise_initial} with the same $\Lambda_{noi}, \Lambda_{ini}$. Note that $\widetilde{W}^{(n)}$ is the discretization of $\widetilde{W}$ defined in the same way as $W^{(n)}$ with respect to the same partition $\pi^{(n)}$. Then for every $\eps>0$ and $\rho\geq 1$, there exists $\delta > 0$ depending on $\eps$ and $\rho$ such that if for every $p \in [1,+\infty)$, one has
	\begin{equation}\label{eq: closed}
	\|f - \widetilde{f}\|_{L_{\omega}^{\rho}L_{x}^{2}} + \sum_{k \in \N} \|V_k - \widetilde{V}_k\|_{L_x^{p}} \lesssim_{p} \delta
	\end{equation}
	then
	\begin{equation}
	\|u_{m}^{(n)} - \widetilde{u}_{m}^{(n)}\|_{L_{\omega}^{\rho}\xX(0,1)} < \eps, 
	\end{equation}
	and the same is true for $\|u^{(n)} - \widetilde{u}^{(n)}\|_{L_{\omega}^{\rho}\xX(0,1)}$. All the proportionality constants are uniform in both $n$ and $m$. 
\end{cor}

\begin{remark}
We will only need some finite choices of $p$ in \eqref{eq: closed}. However, since we assume a priori bounds for $\sum_{k}\|V_{k}\|_{L_x^{\infty}}$ and $\sum_{k}\|\widetilde{V}_{k}\|_{L_x^{\infty}}$, the validity of \eqref{eq: closed} for any $p=p_{0}$ already implies the same is true for all $p>p_{0}$ (with a different proportionality constant depending on $p$). 
\end{remark}

To prove Theorem \ref{th:main}, we split $u^{(n)}-u$ into
\begin{equation} \label{eq:split}
\|u^{(n)} - u\| \leq \|u^{(n)} - u^{(n)}_{m}\| + \|u^{(n)}_{m} - u_{m}\| + \|u_m - u\|, 
\end{equation}
where all the norms are $L_{\omega}^{\rho}\xX(0,1)$. By the construction in \cite{snls_mass_critical} (which we stated in Theorem~\ref{th:snls_mass_critical} above), $u_m \rightarrow u$ as $m \rightarrow +\infty$. Note that once uniform boundedness on $\{u_m^{(n)}\}$ (and hence $\{u_m\}$) is established, the convergence of $u_m$ to the limit $u$ does not use the regularity of the noise any more (see \cite[Section~5]{snls_mass_critical}). Hence the convergence also holds in our situation. It then remains to show the convergence of the first two terms, which will be the material of the following two propositions. The convergence $\unm\rightarrow \un$ follows from the our uniform bound \eqref{eq:unm_uniform_bd} and the aforementioned observation $u^{(n)}_{m} = u^{(n)}$ if the $L_{t}^{5}L_{x}^{10}$ norm of the latter is smaller than $m$. 

\begin{prop} \label{pr:unm_to_un}
	For every initial data $f \in L_{\omega}^{\infty}L_x^2$, we have
	\begin{equation}
	\sup_n \|u^{(n)}_m - u^{(n)}\|_{L_{\omega}^{\rho}\xX(0,1)} \rightarrow 0
	\end{equation}
	as $m \rightarrow +\infty$. 
\end{prop}
\begin{proof}
	Let
	\begin{equation}
	\Omega^{(n)}_{K} = \big\{ \omega \in \Omega: \|u^{(n)}\|_{\xX(0,1)} \geq K \big\}. 
	\end{equation}
	Recall the definitions of $u^{(n)}$ and $u^{(n)}_{m}$ from \eqref{eq:un} and \eqref{eq:unm}. Note that $u^{(n)}_{m} = u^{(n)}$ on $(\Omega^{(n)}_{m})^{c}$. Hence, we have
	\begin{equation}
	\begin{split}
	\|u^{(n)}_{m} - u^{(n)}\|_{L_{\omega}^{\rho}\xX(0,1)} &= \| \1_{\Omega^{(n)}_{m}} (u^{(n)}_{m} - u^{(n)}) \|_{L_{\omega}^{\rho}\xX(0,1)}\\
	&\leq \big( \Pr(\Omega^{(n)}_{m}) \big)^{\frac{1}{2\rho}} \Big( \|u^{(n)}_{m}\|_{L_{\omega}^{2\rho}\xX(0,1)} + \|u^{(n)}\|_{L_{\omega}^{2\rho}\xX(0,1)} \Big). 
	\end{split}
	\end{equation}
	By Theorem~\ref{th:unm_uniform_bd_stable}, we have
	\begin{equation}
	\Pr\big( \Omega^{(n)}_{m} \big) \leq m^{-\rho} \|u^{(n)}\|_{L_{\omega}^{\rho}\xX(0,1)} \lesssim_{\rho} m^{-\rho}
	\end{equation}
	for all $\rho \geq 1$. Hence the claim follows. 
\end{proof}

The convergence of the second term in \eqref{eq:split} is the following proposition.

\begin{prop} \label{pr:unm_to_um}
	For every initial data $f \in L_{\omega}^{\infty}L_{x}^{2}$ and every $m>0$, we have
	\begin{equation}
	\|u^{(n)}_m - u_m\|_{L_{\omega}^{\rho}\xX(0,1)} \rightarrow 0
	\end{equation}
	as $n \rightarrow +\infty$. 
\end{prop}

The proof of the above proposition is another main ingredient of the article. It is in this step we see the Wong-Zakai type convergence. It will be proved in Section~\ref{sec:pr_WZ}. 

The basic idea is that uniform bound \eqref{eq:unm_uniform_bd} allows us to reduce the study of $\un$ to $\unm$, which essentially linearizes the dynamics. This is in particular important since we are working on stochastic problems, where the nonlinearity causes extra difficulties due to loss of integrability in probability space. Wong-Zakai convergence is nontrivial even for linear stochastic Schr\"odinger equations, since the propagator $e^{it\Delta}$ does not have time regularity in $L_x^2$. This is the reason why we need the stability statement (Corollary \ref{cor: stab}) to regularize the initial data and noise. The proof of Proposition~\ref{pr:unm_to_um} is similar to that of Corollary~\ref{cor: stab}.

\subsection{Organization of the article}

According to the above sketch, the proof of the main result (Theorem~\ref{th:main}) will be complete if we prove Theorem~\ref{th:unm_uniform_bd_stable}, Corollary \ref{cor: stab} and Proposition~\ref{pr:unm_to_um}. The rest of the article is thus organized as follows. 

In Section~\ref{sec:prelim}, we present some preliminary bounds and lemmas. We will prove Theorem~\ref{th:unm_uniform_bd_stable} in Section~\ref{sec:unm_bd}, Corollary~\ref{cor: stab} in  Section~\ref{sec:cor_stab}, and finally Proposition~\ref{pr:unm_to_um} Section \ref{sec:pr_WZ}.

\subsection*{Notations}

We introduce some notations that will be frequently used in the rest of the article, including those that have been mentioned above. For any time interval $\iI$, we let
\begin{equation}
\xX_1(\iI) = L_{t}^{\infty}\big(\iI, L_{x}^{2}(\RR)\big)\;, \quad \xX_{2}(\iI) = L_{t}^{5} \big( \iI, L_{x}^{10} \big). 
\end{equation}
We let $\xX = \xX_{1} \cap \xX_2$ in the sense that $\|\cdot\|_{\xX(\iI)} = \|\cdot\|_{\xX_1(\iI)} + \|\cdot\|_{\xX_2(\iI)}$. For $\rho \geq 1$, we write $L_{\omega}^{\rho}\yY = L^{\rho}\big(\Omega, \yY\big)$. We also write $\nN(v) = |v|^{4} v$ for the nonlinearity. 

In what follows, we will fix the (large) discretization parameter $n$. We either obtain bounds that do not depend on $n$, or we compare $u^{(n)}_{m}$ and $u^{(n)}$ with their candidate limits $u_m$ and $u$. In particular, at most one $n$ is involved in this article, and we will never compare with two different $n_1$ and $n_2$. Hence, for points in the partition $\pi^{(n)}$, we drop the superscript $n$ and write $t_{j} = t_{j}^{(n)}$ for simplicity. For $s \in [0,1]$, we let $j(s)$ be the integer in $\{0, \dots, n\}$ such that $s \in [t_{j(s)}, t_{j(s)+1})$. We also let $[s] = t_{j(s)}$ for $s \in [0,1]$. Finally, for $0 \leq j \leq n-1$, we write $\Delta_j W = W(t_{j+1}) - W(t_j)$. 

We will write $A \lesssim B$ if there exists $C$ independent of the quantities $A$ and $B$ such that $A \leq CB$. When such a $C$ depends on some parameter (say $\alpha$), we will write $A \lesssim_{\alpha} B$. In what follows, we do not keep track of explicit dependence on $\rho$, $\Lambda_{ini}$ and $\Lambda_{noi}$, and we will simply write $\lesssim$ instead of $\lesssim_\rho$ for example. Finally, we use $c$ and $C$ to denote constants whose values may change from line to line.

\subsection*{Acknowledgements}

WX acknowledges the support from the Engineering and Physical Sciences Research Council through the fellowship EP/N021568/1. CF thanks Carlos Kenig and Gigliola Staffilani for discussion and encouragement. CF also thanks Yanqi Qiu for discussion on geometry of Banach spaces. Part of this work was done when the authors were at the Universities of Chicago and Oxford respectively, which provide ideal environment for mathematical research.

\section{Preliminaries}\label{sec:prelim}

\subsection{Dispersive and Strichartz estimates}

We start with the by-now standard dispersive and Strichartz estimates. Let $e^{it\Delta}$ be the free propagator of linear Schr\"odinger equation. Then we have the following dispersive and Strichartz estimates. All the bounds below are for dimension one. 

\begin{lem}
	We have
	\begin{equation} \label{eq: dispersive}
	\|e^{it\Delta} f\|_{L^{p}(\RR)} \lesssim  t^{\frac{1}{2}-\frac{1}{p}} \|f\|_{L^{p'}(\RR)}
	\end{equation}
	for every $p \geq 2$ and $p'$ is the conjugate of $p$. 
\end{lem}

We call $(q,r)$ an admissible pair (in dimension one) if $\frac{2}{q} + \frac{1}{r} = \frac{1}{2}$. We have the following Strichartz estimates. 

\begin{lem}
	We have
	\begin{equation} \label{eq: stri}
	\|e^{it\Delta}f\|_{L_{t}^{q}L_{x}^{r}(\RR)} \leq C \|f\|_{L_{x}^{2}}, \quad \Big\| \int_{0}^{t} e^{i(t-s)\Delta} \sigma(s) {\rm d} s \Big\|_{L_{t}^{q}L_{x}^{r}(\iI)} \leq C \|\sigma\|_{L_{t}^{\tilde{q}'}L_{x}^{\tilde{r}'}(\iI)} 
	\end{equation}
	for any admissible pairs $(q,r)$ and $(\tilde{q}, \tilde{r})$, where $\tilde{q}'$ and $\tilde{r}'$ are conjugates of $\tilde{q}$ and $\tilde{r}$. 
\end{lem}

We refer to \cite{cazenave2003semilinear}, \cite{keel1998endpoint} and \cite{tao2006nonlinear} and reference therein for more details.

\subsection{Standard and modified stability for critical NLS}

In this section, we present several stability results for deterministic NLS. We focus on $d=1$ for simplicity, but this part has natural generalizations to higher dimensions. We remark that the estimates below do not rely on the choice of time interval $[0,T]$. We start with the standard and most frequently used stability results for NLS. We refer to Lemmas 3.9 and 3.10 in \cite{colliander2008global} for example. 

\begin{prop}\label{prop: stable}
Suppose $\widetilde{w}$ solves
\begin{equation}\label{eq: pp1}
i \d_t \widetilde{w} +\Delta \widetilde{w}=|\widetilde{w}|^{4} \widetilde{w}+e\;, \quad (t,x) \in [0,T] \times \RR\;, 
\end{equation}
where
\begin{enumerate}
\item $\|\widetilde{w}\|_{L_{t}^{\infty}L_{x}^{2}(0,T)} \leq M$\;,
\item $\|\widetilde{w}\|_{L_{t}^{5}L_{x}^{10}(0,T)} \leq E$\;. 
\end{enumerate}
Then there exists $\eps_{0}>0$ depending on $M$ and $E$ only such that if $w$ solves \eqref{eq:nls_crit} with
\begin{equation}
\|w(0)-\widetilde{w}(0)\|_{L_{x}^{2}}\leq \eps \leq \eps_{0}\;, \quad \|e\|_{L_{t}^{1}L_{x}^{2}(0,T)} \leq \eps \leq \eps_{0}\;, 
\end{equation}
then
\begin{equation}
\|w-\widetilde{w}\|_{\xX(0,T)} \lesssim_{M,E} \eps. 
\end{equation}
This in particular implies
\begin{equation}
\|w\|_{\xX(0,T)}\lesssim_{M,E} 1.
\end{equation}
\end{prop}

While Proposition \ref{prop: stable} is purely perturbative, one can combine it with Dodson's global well-posedness result (Theorem \ref{th:Dodson}) to improve to the following statement. 

\begin{prop} \label{prop: stablework}
Suppose $\tw$ solves
\begin{equation}
i \d_t \widetilde{w} + \Delta \widetilde{w} = |\widetilde{w}|^{4} \widetilde{w}+e
\end{equation}
on $[0,T]$ with $\|\tw(0)\|_{L_{x}^{2}}\leq M$. Then there exists $\eps>0$ depending on $M$ only such that if 
\begin{equation}
\|e\|_{L_{t}^{1}L_{x}^{2}(0,T)} \leq \eps,
\end{equation} 
then
\begin{equation}
\|\widetilde{w}\|_{\xX(0,T)}\lesssim_{M} 1.
\end{equation}
\end{prop}

\begin{remark}
In other words, the assumption for $\|\widetilde{w}\|_{L_{t}^{5}L_{x}^{10}(0,T)}$ in Proposition~\ref{prop: stable} is actually a consequence of the boundedness of $\|\widetilde{w}(0)\|_{L_x^2}$ and the smallness of the perturbation $e$ (depending on the size of $\|\widetilde{w}\|_{L_x^2}$). 
\end{remark}

Proposition~\ref{prop: stablework} implies the following a priori bound. 

\begin{cor} \label{cor: stablebound}
Suppose $w$ solves
\begin{equation}
i \d_t w +\Delta w=|w|^{4}w+e
\end{equation}
on $[0,T]$ with $\|w\|_{L_t^{\infty}L_{x}^{2}}\leq M$ and $\|e\|_{L_{t}^{1}L_{x}^{2}}\leq E$. Then we have
\begin{equation}
\|w\|_{\xX(0,T)} \lesssim_{M} 1+E. 
\end{equation}
\end{cor}

\begin{remark} \label{rem: useful}
The above a priori bound is enough for one to establish the pathwise global well-posedness for \eqref{eq:un}. 
\end{remark}

We finally present a stability argument which will be useful in the study of stochastic dispersive equations, in particular when it is combined with the so-called Da Prato-Debussche method (\cite{DaPrato_Debussche}). 

\begin{prop} \label{prop: snlsstable}
Let $[a,b]$ be an interval and $u,g\in \xX(a,b)$ satisfying $g(a)=0$ and 
\begin{equation}
u(t)=e^{i(t-a)\Delta}u(a)-i\int_{0}^{t}e^{i(t-s)\Delta} \big( |u(s)|^{4}u(s) \big) {\rm d}s + g(t). 
\end{equation}
Then for every $M>0$, there exists $\eta_{M},B_{M}>0$ such that if 
\begin{equation}
\|u\|_{L_{t}^{\infty}L_{x}^{2}(a,b)} \leq M, \quad \|g\|_{L_{t}^{5}L_{x}^{10}(a,b)} \leq \eta_{M},
\end{equation}
then we have
\begin{equation}
\|u\|_{L_{t}^{5}L^{10}_{x}(a,b)}\leq B_{M}.
\end{equation}
\end{prop}

This proposition has played an important role in the construction of the solution to \eqref{eq:snls_crit} (see \cite[Proposition~4.6]{snls_mass_critical}). It was stated in a more complicated way there since we took into account of the truncation. But the estimates are uniform in the truncation parameter $m$, they also work here straightforwardly.  

\begin{remark}
For the stochastic NLS of the current form, it is not clear whether one can write the solution in the form \eqref{eq: pp1} so that the error term $e$ can be well estimated. Instead, one needs to study the stability in the form of its integral version as in Proposition~\ref{prop: snlsstable}. 
\end{remark}

\begin{remark}\label{rem: minden}
We finally point out that we also deal with NLS with time dependent truncated nonlinearity of form $\phi_{m}(\|w\|_{\chi_{2}}(0,t))|w|^{4}w$. Propositions~\ref{prop: stable}, ~\ref{prop: stablework}, \ref{prop: snlsstable} and Corollary~\ref{cor: stablebound} all hold if one replaces the nonlinearity $|w|^{4}w$ by $\phi_{m}(\|w\|_{\chi_{2}}(0,t))|w|^{4}w$, and all the implicit constants involved remain unchanged (in particular they are all uniform in $m$). 
\end{remark}

\subsection{Burkholder inequality}

The Burkholder inequality (\cite{BDG, Burkholder}) is a very useful tool in controlling supremum of martingales.  Because of the summability condition \eqref{eq:summable}, we will only need the following simple version. 

\begin{prop}
Let $B(t)$ be a standard Brownian motion. For every $\rho \in [1,+\infty)$ and $p \in [2,+\infty)$, and every $\sigma$ right-continuous adapted process (to $B$) in $L^{p}$, we have
\begin{equation} \label{eq: burkholder}
\bigg\| \sup_{a,b \in [0,T]}  \Big\|\int_{a}^{b} \sigma(s) {\rm d} B(s) \Big\|_{L_{x}^{p}} \bigg\|_{L_{\omega}^{\rho}} \lesssim_{\rho,p} \Big\|\int_{0}^{T} \|\sigma(s)\|_{L_{x}^{p}}^{2} {\rm d}s \Big\|_{L_{\omega}^{\rho/2}}^{1/2}. 
\end{equation}
\end{prop}

The proof of this inequality, including its more general version involving $\gamma$-Radonifying operators, can be found in \cite[Theorem 2.1]{BP}. We also refer to \cite{Brzezniak, UMD} for more details. We show two typical examples where the bound \eqref{eq: burkholder} is used in this article. Let $W$ be the Wiener process as in the assumption \eqref{eq:summable}. 
\begin{enumerate}
\item A discrete version\footnote{It can be checked as an application of \eqref{eq: burkholder}, but strictly speaking, \eqref{eq: burkholder} is derived from this discrete version.} of \eqref{eq: burkholder}. Let $0<t_{1}<...t_{n}=1$, let $f_{k}$ be a sequence in $L_{x}^{q}$ and $f_{k}\in \mathcal{F}_{t_{k}}$, and let $\frac{1}{p}=\frac{1}{q}+\frac{1}{r}$. Then
\begin{equation}\label{eq: disburkholder}
\begin{aligned}
\|\sum_{k}f_{k}(W(t_{k+1})-W(t_{k}))\|_{L^{\rho}_{\omega}L_{x}^{p}}&\lesssim_{\rho} \|\sum_{k}\|W(t_{k+1})-W({t_{k}})\|_{L_{x}^{q}}^{2}\|f_{k}\|_{p}^{2}\|_{L_{\omega}^{\rho/2}}^{1/2}\\
&\lesssim_{\rho, \lambda_{noi}}\|\sum_{k}\|t_{k+1}-t_{k}\|f_{k}\|_{p}^{2}\|_{L_{\omega}^{\rho/2}}^{1/2}
\end{aligned}
\end{equation}
\item Estimate regarding the Duhamel formula for Schrodinger equation. Let $u(s)\in \mathcal{F_{s}}$ be an adapted process in $L_{x}^{q}$, let $p\geq 2$, and $\frac{1}{p'}=\frac{1}{q}+\frac{1}{r}$ (this simply means $p'\leq q$), then for any $t\in \mathbb{R}$,
\begin{equation}\label{eq: duhburkholder}
\begin{aligned}
&\|\sup_{a,b\in [0,T]}\int_{a,b}e^{i(t-s)\Delta}u(s)ds\|_{L_{\omega}^{\rho}L_{x}^{p}}\\
&\lesssim_{\rho} \sum_{k} \|\int_{0}^{T}\|e^{i(t-s)\Delta}V_{k}u(s)\|_{L_{x}^{p}}^{2}ds\|_{L_{\omega}^{\rho/2}}^{1/2}\\
&\lesssim_{\rho,\noi} \|\int_{0}^{T}(t-s)^{1-\frac{2}{p}}\|u(s)\|_{L_{x}^{q}}^{2}ds\|_{L_{\omega}^{\rho/2}}^{1/2}
\end{aligned}
\end{equation}
In the last step, we have used dispersive estimate \eqref{eq: dispersive}.
\end{enumerate}

\subsection{Kolmogorov's criterion}

We will need the following Kolmogorov's continuity criterion. 

\begin{prop}\label{prop: kc}
Let $q\geq 2$ and $\beta>1/q$. Let $\yY$ be a Banach space, and let $X: [0,T] \rightarrow \yY$ be a stochastic process such that
\begin{equation}
\|X(t)-X(s)\|_{L_{\omega}^{q} \yY} \lesssim |t-s|^{\beta}. 
\end{equation} 
For every $\alpha \in [0,\beta-1/q)$, let $K_{\alpha} = K_{\alpha}(\omega)$ be the random variable defined by
\begin{equation}
K_{\alpha}(\omega) := \sup_{s \neq t} \frac{\|X(t) - X(s)\|_{\yY}}{|t-s|^{\alpha}}. 
\end{equation}
Then we have
\begin{equation}
\|K_{\alpha}\|_{L_{\omega}^{q}} \lesssim_{\alpha,\beta,q} 1. 
\end{equation}
\end{prop}

In particular, for the noise $W$ satisfying Assumption~\ref{as:noise_initial}, applying Burkholder inequality \eqref{eq: burkholder}, \eqref{eq: disburkholder} and the above Kolmogorov's continuity criterion, we see that for all $p \in [1,+\infty]$ and $\alpha<\frac{1}{2}$, we have
\begin{equation} \label{eq: kccontorlosc}
\bigg\|\sup_{s,t \in [0,1]}\frac{\|W(t)-W(s)\|_{L_{x}^{p}}}{|t-s|^{\alpha}} \bigg\|_{L_{\omega}^{\rho}} \lesssim_{\Lambda_\noi, \alpha, \rho} 1. 
\end{equation}

\section{Uniform boundedness of $\{u^{(n)}_{m}\}$ -- proof of Theorem~\ref{th:unm_uniform_bd_stable}} \label{sec:unm_bd}

The aim of this section is to prove the uniform bounds in Theorem~\ref{th:unm_uniform_bd_stable}. For simplicity of presentation, we prove the second bound in \eqref{eq:unm_uniform_bd} only, which corresponds to $m=+\infty$. Uniform-in-$m$ bounds for $\{\unm\}$ can be obtained in essentially exactly the same way. 

\subsection{Overview of the  proof}

We start by recalling some notations that will be used below, as well as introducing some new ones. We fix a large $n$, and will obtain bounds that do not depend on $n$. We assume $n \gg 1$ so that $\|\pi^{(n)}\| \ll 1$. Recall that we fix $n$ throughout, so we write $t_{j} = t_{j}^{(n)}$ for simplicity, and $\Delta_j W = W(t_{j+1}) - W(t_j)$. For $s \in [0,1]$, we let $j(s) \in \{0, \dots, n\}$ such that $s \in [t_{j(s)}, t_{j(s)+1})$. We also write $[s] = t_{j(s)}$. Finally, we write $v = u^{(n)}$ for simplicity. 

We first observe that for every $\rho \in [1,+\infty)$ and $p \in [1,+\infty]$, we have
\begin{equation} \label{eq: sb}
\|W(t)-W(s)\|_{L_{\omega}^{\rho}L_{x}^{p}}\leq \sum_{k} \big\|V_{k} \big( B_{k}(t)-B_{k}(s) \big) \big\|_{L_{\omega}^{\rho}L_{x}^{p}} \lesssim_{\rho,\Lambda_{\noi}} |t-s|^{1/2}. 
\end{equation}
Also, by mass conservation law, we have
\begin{equation}\label{eq: massapriori}
\|v\|_{L_{t}^{\infty}L_{x}^{2}(0,1)}\leq \Lambda_{ini}\lesssim 1.
\end{equation}
Recall that $v=u^{(n)}$ satisfies the classical PDE \eqref{eq: un2}. Expanding $v$ by Duhamel formula at $a \in [0,1]$, we get
\begin{equation}\label{eq: firstexpansion}
v(t) = e^{i(t-a)\Delta} v(a) - i \int_{a}^{t} e^{i(t-s)} \nN\big(v(s)\big) {\rm d} s -i\int_{a}^{t} e^{i(t-s)\Delta} \Big( v(s) \frac{\Delta_{j(s)}W}{t_{j(s)+1}-t_{j(s)}} \Big) {\rm d} s. 
\end{equation}
We now separate out the martingale part in the last term in \eqref{eq: firstexpansion}. To do this, we further expand $v(s)$ via Duhamel formula at time $[s]$. Observe that for any $r\in ([s],s)$, we have $j(r)=j(s)$. Thus, we have the expansion
\begin{equation} \label{eq: secondexpansion}
v(s) = e^{i(s-[s])\Delta} v([s]) - i \int_{[s]}^{s} e^{i(s-r)} \nN(v(r)) {\rm d}r -i \int_{[s]}^{s} e^{i(s-r)\Delta} \left(v(r) \frac{\Delta_{j(s)}W}{t_{j(s)+1}-t_{j(s)}} \right) {\rm d}r. 
\end{equation}
Combining \eqref{eq: firstexpansion} and \eqref{eq: secondexpansion}, we get
\begin{equation} \label{eq: fullexpansion}
\begin{aligned}
v(t) &= e^{i(t-a)\Delta} v(a) - i \int_{a}^{t} e^{i(t-s)\Delta} \nN(v(s)) {\rm d}s\\
&-\int_{a}^{t} e^{i(t-s)\Delta} \left(\frac{\Delta_{j(s)}W}{t_{j(s)+1}-t_{j(s)}} \int_{[s]}^{s} e^{i(s-r)\Delta} \nN(v(r)) {\rm d}r \right) {\rm d}s\\
&-i \int_{a}^{t} e^{i(t-s)\Delta} \left( \frac{\Delta_{j(s)}W}{t_{j(s)+1}-t_{j(s)}} \cdot e^{i(s-[s])\Delta}v([s]) \right) {\rm d}s\\
&- \int_{a}^{t} e^{i(t-s)\Delta} \left[\frac{\Delta_{j(s)}W}{t_{j(s)+1}-t_{j(s)}} \int_{[s]}^{s} e^{i(s-r)\Delta} \left( \frac{\Delta_{j(s)}W}{t_{j(s)+1}-t_{j(s)}} v(r) \right) {\rm d}r \right] {\rm d}s. 
\end{aligned}
\end{equation}
To simplify the above expression, we define the following quantities: 
\begin{equation}
\sS_{\mar}(a,t) = -i\int_{a}^{t} e^{i(t-s)\Delta} \left(\frac{\Delta_{j(s)}W}{t_{j(s)+1}-t_{j(s)}} \cdot e^{i(s-[s])\Delta} v([s])\right) {\rm d}s\;, 
\end{equation}
and
\begin{equation}
\sS_{\qua}(a,t) = - \int_{a}^{t} e^{i(t-s)\Delta} \left[ \frac{\Delta_{j(s)}W}{t_{j(s)+1}-t_{j(s)}} \int_{[s]}^{s} e^{i(s-r)\Delta} \left( \frac{W(t_{j(s)+1})-W(t_{j(s)})}{t_{j(s)+1}-t_{j(s)}}v(r) \right) {\rm d}r \right] {\rm d}s\;,
\end{equation}
and
\begin{equation}
\widetilde{N} \big(v; s\big) = \frac{\Delta_{j(s)}W}{t_{j(s)+1}-t_{j(s)}} \int_{[s]}^{s}e^{i(s-r)\Delta} \nN \big(v(r) \big) {\rm d}r. 
\end{equation}
The expression \eqref{eq: fullexpansion} can then be simplified to
\begin{equation} \label{eq: workingdu}
\begin{split}
v(t) = &e^{i(t-a)\Delta} v(a) - i \int_{a}^{t} e^{i(t-s)\Delta} \nN(v(s)) {\rm d}s\\
&- \int_{a}^{t} e^{i(t-s)\Delta} \sn {\rm d}s + \sS_{\mar}(a,t) + \sS_{\qua}(a,t). 
\end{split}
\end{equation}
We will view $\ssm$ and $\ssq$ as source terms and $\sn$ as a small perturbation. Roughly speaking, we want to get supreme-in-$a$ bounds for $\ssm$ and $\ssq$. For $\ssm$, we will explore the martingale structure. For $\ssq$, we will use the fact $\sum_{j}\|W(t_{j+1})-W_{t_{j}}\|^{2}\sim \sum_{j}|t_{j+1}-t_{j}| \lesssim 1$ as long as $\|\cdot\|$ is a reasonable $L^p$ space norm which will be specified later. As for $\sn$, we observe that at least formally, its norm is of size $\|W(t_{j+1})-W_{t_{j}}\| \cdot \| \nN\big(v(s)\big)\|\sim |t_{j+1}-t_{j}|^{1/2}\|\nN(v(s))\|$, which can be treated perturbatively via a bootstrap argument. Again, all the norms and bounds will be specified later. 

We fix a small parameter $\eta>0$ (independent of $n$), and separate two situations depending on whether $\|\Delta_{j} W\| < \eta$ or not, and the norm used will be specified later. In the (very rare) event it is bigger than $\eta$,  we view $v \Delta_j W$ as a perturbative term, and use brute force stability arguments for \eqref{eq: un2} together with a quantitative estimate implying such an event is rare. In the more common situation when $\|\Delta_j W\| < \eta$, we use maximal estimates for $\sS_{\max}$ and $\sS_{\qua}$ to combine sub-intervals in $\pi^{(n)}$ and run bootstrapping arguments on these joint intervals. 

We start with the maximal estimates for $\ssm$ and $\ssq$.

\subsection{Control of source term $\ssm$, $\ssq$ }

Let $\sS_{\max}^{*}$ and $\sS_{\qua}^{*}$ be defined as
\begin{equation} \label{eq:smmax}
	\sS^{*}_{\mar}(t) :=  \sup_{0 \leq \tau \leq t}   \Big\| \int_{0}^{\tau} e^{i(t-s)\Delta} \left( \frac{\Delta_{j(s)}W}{t_{j(s)+1}-t_{j(s)}} e^{i(s-[s]\Delta)}v([s]) \right) {\rm d} s \Big\|_{L_x^{10}}\;, 
	\end{equation}
	and
	\begin{equation}\label{eq: sqmax}
	\sS^{*}_{\qua}(t) := \int_{0}^{t}\int_{[s]}^{s}\Big\|e^{i(t-s)\Delta} \left[ \frac{\Delta_{j(s)}W}{t_{j(s)+1}-t_{j(s)}}e^{i(s-r)\Delta} \left( \frac{\Delta_{j(s)}W}{t_{j(s)+1}-t_{j(s)}}v(r) \right) \Big\|_{L_{x}^{10}} {\rm d} r \right] {\rm d}s\;. 
	\end{equation}
It is clear that
\begin{equation}\label{eq: estimatestragiht}
\sup_{a \in [0,t]} \|\ssm(a,t)\|_{L_{x}^{10}} \leq \sm(t)\;, \quad \sup_{a \in [0,t]}\|\ssq(a,t)\|_{L_{x}^{10}} \leq \sq(t)\;. 
\end{equation}
The main technical bound in this subsection is the following lemma. 

\begin{lem}\label{le:source_max}
For every $\rho\geq 1$, we have
\begin{equation}\label{eq: prepareforsplit}
\|\sm(t)\|_{L_{\omega}^{\rho}L_{t}^{5}(0,1)}\lesssim_{\rho} 1\;, \quad \|\sq(t)\|_{L_\omega^{\rho}L_{t}^{5}(0,1)}\lesssim_{\rho} 1\;, 
\end{equation}
where we have omitted the dependence on $\Lambda_\ini$ and $\Lambda_\noi$ for notational simplicity. 
\end{lem}

We only need to prove the above lemma for $\rho\geq 5$. By Minkowski inequality and that our interval is $[0,1]$ (having finite length), we have the embedding
\begin{equation}\label{eq: min}
L_{t}^{\infty}L_{\omega}^{\rho}\hookrightarrow L_{t}^{5}L_{\omega}^{\rho}\hookrightarrow L_{\omega}^{\rho}L_{t}^{5}
\end{equation}
Hence, \eqref{eq: prepareforsplit} will follow from the bounds
\begin{equation}\label{eq: p1es}
\sup_{t \in [0,1]} \|\sS^*_\mar(t)\|_{L^{\rho}_{\omega}}\lesssim 1
\end{equation}
and
\begin{equation}\label{eq: p2es}
\sup_{t \in [0,1]} \|\sS^*_\qua(t)\|_{L_{\omega}^{\rho}}\lesssim 1\;. 
\end{equation}
Fix $t \in [0,1]$. Before we start the proof of \eqref{eq: p1es} and \eqref{eq: p2es},  we first prove the following technical lemma which handles fluctuation in any small interval $[t_{j}, t_{j+1}]$. 

\begin{lem}\label{lem: smallosc}
For any $\kappa < \frac{1}{10}$, we have
\begin{equation}\label{eq: noworrytail}
\sup_{0 \leq j \leq n-1} \left\|\sup_{\tau\in [0,t]} \left\| \int_{[\tau]}^{\tau} e^{i(t-s)\Delta} \left(\frac{\Delta_{j(s)}W}{t_{j(s)+1}-t_{j(s)}} \cdot e^{i(s-[s])\Delta} v(s) \right) {\rm d}s \right\|_{L_{x}^{10}} \right\|_{L_{\omega}^{\rho}} \lesssim_{\rho,\kappa} \|\pi^{(n)}\|^{\frac{1}{10}-\kappa}\;.
\end{equation}
\end{lem}
\begin{proof}
Note that for $[\tau] \leq s \leq \tau$, we have $[\tau]=t_{j(s)}$ and $\tau \in [t_{j(s)}, t_{j(s)+1})$. Thus, we have
\begin{equation}\label{eq: qq}
\begin{aligned}
&\left\|\int_{[\tau]}^{\tau}e^{i(t-s)\Delta} \left(\frac{\Delta_{j(s)}W}{t_{j(s)+1}-t_{j(s)}}e^{i(s-[s])\Delta}v(s) \right) {\rm d}s \right\|_{L_{x}^{10}}\\
\lesssim &\int_{[\tau]}^{\tau} \left\|e^{i(t-s)\Delta} \left(\frac{\Delta_{j(s)}W}{t_{j(s)+1}-t_{j(s)}} \cdot e^{i(s-[s])\Delta} v(s) \right)\right\|_{L_{x}^{10}} {\rm d}s\\
\lesssim &\int_{[\tau]}^{\tau} (t-s)^{-\frac{2}{5}} \left\|\frac{\Delta_{j(s)}W}{t_{j(\tau)+1}-t_{j(\tau)}} \cdot e^{i(s-[s])\Delta}v(s) \right\|_{L_{x}^{10/9}} {\rm d}s\\
\lesssim &\frac{1}{t_{j(s)+1}-t_{j(s)}}\int_{[\tau]}^{\tau} (t-s)^{-\frac{2}{5}} \|\Delta_{j(s)}W\|_{L_{x}^{5/2}} {\rm d}s\\
\lesssim &\frac{\|\Delta_{j(\tau)}W\|_{L_{x}^{5/2}}}{(t_{j(\tau)+1}-t_{j(\tau)})^{2/5}}\;. 
\end{aligned}
\end{equation}
Here, we applied the dispersive estimate \eqref{eq: dispersive} in the third line above, and also used $\|e^{i(s-[s])\Delta}v(s)\|_{L_{x}^{2}} \lesssim 1$. In the fourth line, we have applied H\"older $\|fg\|_{L_{x}^{10/9}} \lesssim \|f\|_{L_{x}^{5/2}}\|g\|_{L_{x}^{2}}$. The claim \eqref{eq: noworrytail} then directly follows from \eqref{eq: kccontorlosc} by taking $L_{\omega}^{\rho}$-norm on both sides and employing the definition that $\|\pi^{(n)}\| = \sup_j |t_{j+1} - t_j|$. 
\end{proof}

We are now ready to prove the bound for $\sS_{\mar}^{*}$. 

\begin{proof}[Proof of \eqref{eq: p1es}]
Splitting the integral in $[0,\tau]$ into $[0,[\tau]]$ and $[[\tau],\tau]$, we have
\begin{equation}
\begin{aligned}
&\int_{0}^{\tau} e^{i(t-s)\Delta} \left(\frac{\Delta_{j(s)}W}{t_{j(s)+1}-t_{j(s)}} \cdot e^{i(s-[s]\Delta)} v([s]) \right) {\rm d}s\\
=&\int_{0}^{[\tau]}e^{i(t-s)\Delta} \left( \dws e^{i(s-[s]\Delta)}v([s]) \right) {\rm d}s\\
+& \int_{[\tau]}^{\tau} e^{i(t-s)\Delta} \left( \dws e^{i(s-[s]\Delta)}v([s]) \right) {\rm d}s\;. 
\end{aligned}
\end{equation}
The second term could be controlled directly via Lemma \ref{lem: smallosc}. Hence, we only need to prove the bound (for the first term)
\begin{equation}\label{eq: intermar}
\left\| \sup_{\tau} \left\| \int_{0}^{[\tau]}e^{i(t-s)\Delta} \left(\dws \cdot e^{i(s-[s]\Delta)}v([s]) \right) {\rm d}s \right\|_{L_{x}^{10}} \right\|_{L_{\omega}^{\rho}}
\lesssim 1\;. 
\end{equation}
Note that
\begin{equation}
\int_{0}^{[\tau]}e^{i(t-s)\Delta} \left( \dws e^{i(s-[s]\Delta)}v([s]) \right) {\rm d}s =\sum_{j=0}^{j(\tau)} \frac{\Delta_j W}{t_{j+1} - t_{j}} \int_{t_{j-1}}^{t_{j}} e^{i(s-t_j)\Delta}v(t_j) {\rm d}s
\end{equation}
is a discrete martingale  in $L_{x}^{10}$, so by Burkholder inequality \eqref{eq: burkholder}, we have
\begin{equation}\label{eq: qq1}
\begin{aligned}
\left\| \sup_{\tau} \left\|\int_{0}^{[\tau]}e^{i(t-s)\Delta} \left(\dws e^{i(s-[s]\Delta)}v([s]) \right) {\rm d} s \right\|_{L_{x}^{10}} \right\|_{L_{\omega}^{\rho}}^{\rho}\\
\lesssim \left\| \sum_{j} \left\|  \int_{t_{j}}^{t_{j+1}} e^{i(t-s)\Delta} \left( \dws e^{i(s-[s])\Delta} v(s) \right) {\rm d}s \right\|^{2}_{L_{x}^{10}} \right\|_{L_{\omega}^{\rho/2}}^{\rho/2}\;. 
\end{aligned}
\end{equation}
Using the dispersive estimate \eqref{eq: dispersive} and arguing similarly as in \eqref{eq: qq}, we get
\begin{equation}\label{eq: qq2}
\left\| \int_{t_{j}}^{t_{j+1}} e^{i(t-s)\Delta} \left( \dws e^{i(s-[s])\Delta} v(s) \right) {\rm d}s \right\|^{2}_{L_{x}^{10}} \lesssim \frac{\|\Delta_j W\|_{L_{x}^{5/2}}}{t_{j+1}-t_{j}} \int_{t_{j}}^{t_{j+1}} (t-s)^{-2/5} {\rm d}s\;. 
\end{equation}
Plug \eqref{eq: qq2} into \eqref{eq: qq1}, we see that \eqref{eq: intermar} will follow if we can prove
\begin{equation}
\left\| \sum_{j} \left( \frac{\|\Delta_j W\|_{L_{x}^{5/2}}}{t_{j+1}-t_{j}} \int_{t_{j}}^{t_{j+1}} (t-s)^{-\frac{2}{5}} {\rm d}s \right)^{2} \right\|_{L_{\omega}^{\rho/2}}^{\rho/2} \lesssim 1.
\end{equation}
Indeed, using triangle inequality to insert $L_{\omega}^{\rho/2}$-norm inside the sum and then employing \eqref{eq: sb}, we have 
\begin{equation}\label{eq: pp4}
\begin{aligned}
&\phantom{111}\left\| \sum_{j} \left( \frac{\|\Delta_j W\|_{L_{x}^{5/2}}}{t_{j+1}-t_{j}} \int_{t_{j}}^{t_{j+1}}(t-s)^{-2/5} {\rm d}s \right)^{2} \right\|_{L_{\omega}^{\rho/2}}^{\rho/2}\\
&\lesssim_{\rho} \sum_{j}\frac{1}{t_{j+1}-t_{j}} \left( \int_{t_{j}}^{t_{j+1}} (t-s)^{-2/5} {\rm d}s \right)^{2}\\
&\lesssim \sum_{j=0}^{n-1} \int_{t_{j}}^{t_{j+1}}(t-s)^{-4/5} {\rm d}s = \int_{0}^{1}(t-s)^{-4/5} {\rm d}s \lesssim 1, 
\end{aligned}
\end{equation}
where we have used H\"older to create a whole power of $t_{j+1}-t_{j}$ in the third inequality. This completes the proof. 
\end{proof}

\begin{remark}\label{remark: extrasmall1}
In the last step of \eqref{eq: pp4}, we used $\int_{0}^{1}(t-s)^{-4/5} {\rm d}s \lesssim 1$. If we work on an small interval $[a,b]$ rather than $[0,1]$, we will gain a small power of form $(b-a)^{\gamma}$. In this situation we would have $\gamma=1/5$, but in general this term is subcritical and one can gain a positive power of $b-a$. 
\end{remark}

Now we turn to the control of $\sq$. 

\begin{proof}[Proof of \eqref{eq: p2es}]
Unlike \eqref{eq: p1es} which relies on the martingale structure, the proof of \eqref{eq: p2es} is more straight forward. Proceeding similarly as in \eqref{eq: qq}, for $t_{j} \leq r \leq s \leq t_{j+1}$, we have
\begin{equation}
\begin{aligned}
&\left\| e^{i(t-s)\Delta} \left[ \dws \cdot e^{i(s-r)\Delta} \left( \dws v(r) \right) \right] \right\|_{L_{x}^{10}}\\
\lesssim &(t-s)^{-\frac{2}{5}} \frac{\left\| \Delta_{j(s)} W \right\|_{L_{x}^{5/2}} \left\| \Delta_{j(s)} W \right\|_{L_{x}^{\infty}}}{\big( t_{j(s)+1} - t_{j(s)} \big)^{2}}\;.
\end{aligned}
\end{equation}
We thus get
\begin{equation}\label{eq: pp3}
\|\sq(t)\|_{L_{\omega}^{\rho}}
\lesssim \left\| \sum_{j} \frac{\|\Delta_j W\|_{L_x^{5/2}} \|\Delta_j W\|_{L_x^\infty}}{t_{j+1}-t_{j}} \int_{t_{j}}^{t_{j+1}} (t-s)^{-\frac{2}{5}} {\rm d}s \right\|_{L_{\omega}^{\rho}}  \lesssim_\rho 1.
\end{equation}
This completes the proof. 
\end{proof}

\begin{remark}\label{remark: exs2}
Similar as in Remark \ref{remark: extrasmall1}, one can gain a positive power of $b-a$ if one is working in a small interval $b-a$ instead of all of $[0,1]$. 
\end{remark}

\subsection{A crude upper bound -- controlling large oscillations}

We give a crude bound on $\|v\|_{\xX(t_j, t_{j+1})}$. This bound will used only when the size of $\Delta_j W$ is not small in that interval. 
 
\begin{lem}\label{le: lemlosc}
For every $0 \leq j \leq n-1$, we have
\begin{equation}
\|v\|_{\xX(t_{j},t_{j+1})} \lesssim 1 + \|\Delta_j W\|_{L_{x}^{\infty}}\;. 
\end{equation}
\end{lem}
\begin{proof}
Note that
\begin{equation}
\left\| v \cdot \frac{\Delta_j W}{t_{j+1}-t_{j}} \right\|_{L_{t}^{1}L_{x}^{2}} \lesssim \|\Delta_j W\|_{L_{x}^{\infty}}.
\end{equation}
The claim then follows from \eqref{eq: massapriori} and Corollary~\ref{cor: stablebound}.
\end{proof}

\begin{remark}
Although Lemma \ref{le: lemlosc} is true for all $j$ and all realizations of $W$, we will use it only for the rare event when $\|W(t_{j+1})-W_{t_{j}}\|_{L_x^\infty}$ has size of at least order $1$. 
\end{remark}

\begin{remark}\label{rem: lpnot}
We discuss how one can use $L_x^p$-norm of $\Delta_j W$ for $p<+\infty$, and not assuming $L_x^\infty$ norm. To do this, we can replace the $L_{t}^{1}L_{x}^{2}$ norm by any $L_{t}^{q'}L_{x}^{r'}$ in Propositions~\ref{prop: stable}, \ref{prop: stablework} and Corollary~\ref{cor: stablebound}. For example, if one uses the pair $(5,10)$, then the same argument gives
\begin{equation}
\|v\|_{\xX(t_{j},t_{j+1})} \lesssim 1 + \frac{\| \Delta_j W \|_{L_{x}^{5/2}}}{(t_{j+1}-t_{j})^{1/4}}.
\end{equation}
which is still good for us since $\Delta_j W$ has typical size $|t_{j+1}-t_{j}|^{1/2}$.
\end{remark}

\subsection{Deriving the desired bound}

We are ready to prove the desire bound for $v$, that is, 
\begin{equation}\label{eq: desired}
\|v\|_{L_{\omega}^{\rho}\xX_{2}(0,1)} \lesssim_{\rho} 1
\end{equation}
The $\xX_{1}$ part is automatic since we have mass conservation law \eqref{eq: massapriori}. We fix a small constant $\eta>0$ whose value (independent of $n$) will be specified later. 

Recall that the partition $0=t_{0}<t_{1}<...t_{n}=1$ is fixed throughout. For every $\omega$, we partition the interval $[0,1]$ into a union of almost disjoint subintervals (only their boundary points can intersect), and divide those subintervals into two types as follows. 

First, type-A intervals are selected from the original partition $\pi^{(n)}$. An interval $[t_j, t_{j+1}]$ is called a \textit{type-A interval} if $\|\Delta_j W\|_{L_x^\infty} \geq \eta$. The remaining set $[0,1] \setminus \cup_{\text{type A}} (t_j, t_{j+1})$ (where we only take the interior of type-A intervals here) is a union of disjoint closed intervals, each of the form $[t_{j}, t_{j'}]$ where $t_j$ and $t_{j'}$ are end and beginning points of two consecutive type-A intervals. We now further partition the remaining set into type B intervals $\{[a_l, b_l]\}_l$ as follows. For any such (maximal) interval $[t_j, t_{j'}]$ in the remaining set, let $a_1 = t_j$. Suppose $[a_l, b_l] \subset [t_{j}, t_{j'}]$ is defined, let $a_{l+1} := b_l$, and define
\begin{equation}
	b_{l+1} := t_{j'} \wedge \inf \Big\{\tau \geq a_{l+1}: \|\sS_{\mar}^{*}\|_{L_t^5 (a_{l+1},\tau)} + \|\sS_{\qua}^{*}\|_{L_t^5(a_{l+1}, \tau)} = \eta \Big\}\;.
\end{equation}
This procedure necessarily ends with a finite $l$. We repeat this procedure for every maximal closed interval $[t_j, t_{j'}] \in [0,1] \setminus \cup_{type A} [t_j, t_{j+1}]$. In this way, we have partitioned the remaining set into a disjoint union of intervals $[a_l, b_l]$ (only boundary points can overlap) such that
\begin{equation}
	[0,1] \setminus \bigcup_{\text{type A}} (t_{j}, t_{j+1}) = \cup_{l} [a_l, b_l]\;, \qquad (a_l, b_l) \cap (a_{l'}, b_{l'}) = \emptyset \; \; \text{if} \; \; l \neq l'\;.
\end{equation}
Here, we have abused the index $l$ by relabelling all such intervals. By construction, one also has the estimate
\begin{equation} \label{eq: smallness}
	\|\sS_{\mar}^{*}\|_{L_t^5(a_l, b_l)} + \|\sS_{\qua}^{*}\|_{L_t^5(a_l, b_l)} \leq \eta\;
\end{equation}
for every $l$. For convenience, we further separate type-B intervals into subtypes B-I and B-II as follows: 
\begin{itemize}
\item We call $[a_l, b_l]$ type B-I if
\begin{equation}\label{eq: nottoosmall}
\|\sq\|_{L_{t}^{5}(a_{l},b_{l})}+\|\sm\|_{L_{t}^{5}(a_{l},b_{l})} \geq \frac{\eta}{2}\;.
\end{equation}
\item All the remaining type-B intervals are of type B-II. 
\end{itemize}
Thus, we have defined a random variable $\omega \mapsto J(\omega)$ where $J=J(\omega)$ is the number of type-B intervals. 

We now state two lemmas to summarize the properties of those partitions. First,we claim $J$ can not be too large in average sense. 
\begin{lem}\label{lem: finitenumber}
\begin{equation}
\|J^{1/5}\|_{L_{\omega}^{\rho}}\lesssim_{\rho,\eta} 1
\end{equation}
\end{lem}

Second, when the $\eta$ is chosen small enough, (such smallness only depends on $\Lambda_{ini}$ and the fact $\|\pi\|$ is small enough), we have
\begin{lem}\label{lem: smalldynamic}
If in some interval $[a_{l},b_{l}]$ so that \eqref{eq: smallness} holds, then we have the (deterministic) estimate
\begin{equation}
\|v\|_{\xX_{2}(a_{l},b_{l})} \lesssim 1\;. 
\end{equation}
\end{lem}

Assuming Lemma \ref{lem: finitenumber}, \ref{lem: smalldynamic} at the moment, we conclude the proof of \eqref{eq: desired}.
\begin{proof}[Proof of \eqref{eq: desired} assuming Lemmas~\ref{lem: finitenumber} and~\ref{lem: smalldynamic}]
If some interval is of type-A, we apply Lemma \ref{le: lemlosc}.
All the intervals of type-B are partitioned into sub-intervals $[a_{l},b_{l}]$ in which \eqref{eq: smallness} hold, and we apply Lemma \ref{lem: smalldynamic}. To summarize, we derive
\begin{equation}\label{eq: conclude}
\begin{aligned}
\|v\|_{\xX_{2}[0,1]}\leq &\|v\|_{L_{t}^{5}L_{x}^{10}(\cup_{l}[a_{l},b_{l}])} + \sum_{[t_{j},t_{j+1}] \text{ type-A}} \|v\|_{L_{t}^{5}L_{x}^{10}(t_{j},t_{j+1})}\\
\lesssim & J^{\frac{1}{5}} + \sum_{j} \1_{\{\|\Delta_j W\|_{L_{x}^{\infty}} \geq \eta\}} (\omega) \left( 1 + \|\Delta_j W\|_{L_{x}^{\infty}} \right)\;.
\end{aligned}
\end{equation}
Thus, we derive 
\begin{equation}
\|v\|_{L_{\omega}^{\rho}\xX_{2}(0,1)} \lesssim_\rho \| J^{\frac{1}{5}} \|_{L_{\omega}^{\rho}} + \sum_{j} \left\| \1_{\{\|\Delta_j W\|_{L_{x}^{\infty}} \geq \eta\}} (\omega) \left( 1 + \|\Delta_j W\|_{L_{x}^{\infty}} \right) \right\|_{L_{\omega}^{\rho}}\;. 
\end{equation}
The first term is controlled by Lemma \ref{lem: finitenumber}. For the second one, using the fact that for each $j$, we have
\begin{equation}
\mathbb{P}(\|W(t_{j+1})-W(t_{j})\|_{L_{x}^{\infty}} \sim \alpha)\lesssim   e^{-c|\frac{\alpha}{t_{j+1}-t_{j}}|^{2}},
\end{equation}
where $c$ depends on $\Lambda_\ini$ only. Thus, we derive
\begin{equation}
\begin{aligned}
&\sum_{j} \left\| \chi_{\{\|\Delta_j W\|_{L_{x}^{\infty}} \geq \eta\}} (\omega) \left( 1 + \|\Delta_j W\|_{L_{x}^{\infty}} \right) \right\|_{L_{\omega}^{\rho}}\\
\lesssim_{\rho,\eta} &\sum_{j}[e^{-\frac{c}{2}|\frac{\eta}{t_{j+1}-t_{j}}|^{2}}]^{1/\rho} \lesssim_{\rho,\eta} \sum_{j}|t_{j+1}-t_{j}|^{2} \lesssim \|\pi^{(n)}\| \lesssim 1.
\end{aligned}
\end{equation}
Note that since $\eta > 0$ is fixed and independent of $n$ (but depending on $\Lambda_\ini$ and $\Lambda_\noi$ only), we can thus conclude the proof. 
\end{proof}

We are left with the proofs of Lemmas~\ref{lem: finitenumber} and~\ref{lem: smalldynamic}. We first prove Lemma~\ref{lem: finitenumber}.

\begin{proof}[Proof of Lemma \ref{lem: finitenumber}]
Let $J_{1}$ be the number of intervals of Type B-I, and $J_{2}$ be the number of intervals of Type B-II. And  $J=J_{1}+J_{2}$.
Note that
\begin{equation}
 \sum_{
 (a_{l},b_{l}) \text{ Type B-I}} \Big( \|\sq\|_{L_{t}^{5}(a_{l},b_{l})}^{5}+\|\sm\|_{L_t^5(a_l,b_l)}^{5} \Big) \sim_{\eta}J_{1}
\end{equation}
Applying Lemma~\ref{le:source_max}, the desired estimate follows for $J_{1}$.

On the other hand, we have
\begin{equation}
J_{2} \leq \# \; \text{of type A intervals}\;.
\end{equation}
Taking $L^{\rho}$ on both side, and recalling that the probability of $[t_{j},t_{j+1}]$ being Type A is $\lesssim_{\eta} |t_{j+1}-t_{j}|^{2}$, the desired estimate for $J_{2}$ follows. We can then conclude Lemma~\ref{lem: finitenumber}. 
\end{proof}

Finally, we present the proof of Lemma~\ref{lem: smalldynamic}. Implicitly, part of the  proof is in the same spirit of so-called Da Prato-Debussche trick.

\begin{proof} [Proof of Lemma~\ref{lem: smalldynamic}]
	
Let us fix $l$ and denote $[a_{l}, b_{l}]$ by $[a,b]$. We assume without loss of generality that $a$ is a point in the partition $\pi^{(n)}$ so that $[a]=a$. If not, then by our definition of $[a_{l},b_{l}]$, the interval $[t_{j(a)},t_{j(a)+1}]$ is not of type-A, and we can apply Lemma~\ref{le: lemlosc} to control the dynamics of $v$ in $[t_{j(a)},t_{j(a)+1}]$, and we only need to study the dynamics of $v$ in $[t_{j(a)+1},b]$, which satisfies all the assumption for $[a_{l}, b_{l}]$, which we will use below. Similarly, we also assume $b=[b]$.

Recall that $v$ satisfies \eqref{eq: workingdu} for $t \in [a,b]$. Let
\begin{equation}
\begin{aligned}
h_{1}(t):=&- \int_{a}^{t} e^{i(t-s)\Delta} \sn {\rm d} s\\
=&-\int_{a}^{t} e^{i(t-s)\Delta} \left(\dws \int_{[s]}^{s}e^{i(s-r)\Delta}\nN(v(r))dr\right) {\rm d}s\;, 
\end{aligned}
\end{equation}
and $h_{2}(t):=\sS_{\mar}(a,t) + \sS_{\qua}(a,t)$. By the choice of $[a,b]$ and \eqref{eq: smallness}, we observe that
\begin{equation}
h_{2}(a )= h_{1}(a) = 0 \quad \text{and} \quad \|h_{2}(a)\|_{L_{t}^{5}L_{x}^{10}(a,b)} \lesssim \eta.
\end{equation}
If there were no $h_{1}$ term, then Lemma~\ref{lem: smalldynamic} follows from the modified stability Proposition~\ref{prop: snlsstable}.

In our current situation, we treat $h_{1}$ perturbatively. To do this, we will follow a bootstrap scheme. Let $M$ in Proposition~\ref{prop: snlsstable} be the same as $\Lambda_\ini$, hence the corresponding $B_M$ and $\eta_M$ in Proposition~\ref{prop: snlsstable} depends on $\Lambda_\ini$ only. We first assume the following bootstrap lemma. 

\begin{lem} \label{lem: bootstrap}
For sufficiently small $\eta$, if for $[a,T] \subset [a,b]$, one has the bootstrap assumption
\begin{equation}\label{eq: bootstrapas}
\|v\|_{\xX_2(a,T)} \leq 2B_{M}\;, 
\end{equation}
then one has the bootstrap estimate
\begin{equation}
\|v\|_{\xX_2(a,T)} \leq B_{M}\;. 
\end{equation}
\end{lem}

Lemma~\ref{lem: smalldynamic} follows from Lemma~\ref{lem: bootstrap} by standard continuity argument. 
\end{proof}

It then remains to prove Lemma~\ref{lem: bootstrap}. 

\begin{proof} [Proof of Lemma~\ref{lem: bootstrap}]
	
We first choose $\eta$ small enough such that $\|h_{2}(a)\|_{\xX_2(a,b)} \leq \eta_{M}/10$ (we will make $\eta$ even smaller if necessary). Lemma~\ref{lem: bootstrap} follows from Proposition~\ref{prop: snlsstable} if we can show that \eqref{eq: bootstrapas} implies 
\begin{equation}\label{eq: bootstrapsmallness}
\|h_{1}\|_{\xX_2(a,b)}\leq \eta_{M}/10.
\end{equation}
By Strichartz estimate \eqref{eq: stri}, we have
\begin{equation}\label{eq: redobystri}
\|h_{1}\|_{\xX_2(a,t)} \lesssim \left\| \dws \int_{[s]}^{s} e^{i(s-r)\Delta} \nN\big(v(r)\big) {\rm d}r \right\|_{L_{s}^{1}L_{x}^{2}(a,t)}\;. 
\end{equation}
Note that since we assume $a=[a]$ and $b=[b]$, the $[t_{j(s)},t_{j(s)}+1]$ in the integrand is always in $[a,b]$, also note that  $[t_{j(s)},t_{j(s)}+1]$ must be of Type-B.

Note that we have the pointwise estimate
\begin{equation}\label{eq: pe}
\begin{aligned}
&\phantom{11} \left\| \dws \int_{[s]}^{s} e^{i(s-r)\Delta} \nN\big(v(r) \big) {\rm d}r \right\|_{L_{x}^{2}}\\
&\lesssim \frac{\|\Delta_{j(s)} W\|_{L_x^{\infty}}}{t_{j(s)+1}-t_{j(s)}} \cdot \|\nN(v)\|_{L_{t}^{1}L_{x}^{2}([s],s)} \lesssim \frac{\eta}{t_{j(s)+1}-t_{j(s)}} \|v\|_{\xX_2(t_{j(s)},t_{j(s)+1})}^{5}.
\end{aligned}
\end{equation}
Hence, by \eqref{eq: redobystri} and \eqref{eq: pe}, we conclude that
\begin{equation}
\|h_{1}(t)\|_{L_{t}^{1}L_{x}^{2}}\lesssim \sum_{[t_{j},t_{j+1}]\subset [a,b]}\eta \|v\|_{L_{t}^{5}L_{x}^{10}[t_{j},t_{j}+1]}^{5}\lesssim \eta B_{M}^{5}
\end{equation}
We have \eqref{eq: bootstrapsmallness} when $\eta$ is small enough. This completes the proof. 
\end{proof}

\begin{remark} \label{rem:Lp}
Before we end this subsection, we give a brief sketch on how one can modify the above arguments so that one does not need closeness in $L_x^\infty$ to prove stability. As already mentioned in Remark \ref{rem: lpnot}, one can replace the $\|\Delta_j W\|_{L_{x}^{\infty}}$ in Lemma~\ref{le: lemlosc} by $\frac{\|\Delta_j W\|_{L_x^{5/2}}}{(t_{j+1} - t_j)^{1/4}}$. Later, one may define an interval to be of type A if $\frac{\|\Delta_j W\|_{L_x^{5/2}}}{(t_{j+1} - t_j)^{1/4}} \geq \eta$. And finally, in the proof of Lemma~\ref{lem: bootstrap}, one replaces the $L_{t}^{1}L_{x}^{2}$-norm by $L_{t}^{5/4}L_{x}^{10/9}$, which is the dual of the Strichartz pair $(5,10)$. 
\end{remark}

\section{Proof of Corollary~\ref{cor: stab}} \label{sec:cor_stab}

\subsection{Overview of the proof}

We first point out, it is very natural that if one can prove Theorem \ref{th:unm_uniform_bd_stable}, then one can prove a stability result as in Corollary \ref{cor: stab}. It may be of some concern since in Corollary~\ref{cor: stab}, we only assume closeness between $V_{k}$ and $\widetilde{V}_{k}$ in $L_{x}^{p}$ for $p<\infty$ (and an a priori control of $L^{\infty}$), while the proof of Theorem \ref{th:unm_uniform_bd_stable} does use $L_{x}^{\infty}$-norm in several places. However, we use $L^{\infty}_{x}$-norm of $W$ only for the convenience of presentation, and as explained in Remark~\ref{rem:Lp}, one only needs a little modification of the arguments to turn the assumption from $L_x^\infty$ into $L_x^p$. The only place not covered by that remark is in \eqref{eq: pp3}, where $\|\Delta_j W\|_{L_x^\infty}$ also appear. But since the quantity $\|\Delta_{j}W\|_{L_{x}^{\infty}}$ in that expression was multiplied by $\|\Delta_{j}W\|_{L_{x}^{5/2}}$, and since Corollary~\ref{cor: stab} assumes closeness in $L_x^p$ but boundedness in $L_x^\infty$, this is already good enough for stability arguments to go through. We emphasize here again that the assumption in Corollary~\ref{cor: stab} includes a priori bound in $L_x^\infty$, but closeness in $L_x^p$ for $p<+\infty$. 

Due to the above discussion and for simplicity and conciseness of numeric, we will only present the proof of Corollary \ref{cor: stab} replacing \eqref{eq: closed} by the following stronger assumption
\begin{equation}\label{eq: enhanced}
\|f - \widetilde{f}\|_{L_{\omega}^{\rho}L_{x}^{2}} + \sum_{k \in \N} \|V_k - \widetilde{V}_k\|_{L^{1}_{x}\cap L_{x}^{\infty}}<\delta.
\end{equation}
Fix $\eps$ and $\rho$, and  arguing similarly as Proposition \ref{pr:unm_to_un}, one can show that
\begin{equation}\label{eq: mgood}
\sup_n \|\un_{m}-\un\|_{L_{\omega}^{\rho}\xX([0,1])}\leq \eps/10
\end{equation}
for all sufficiently large $m$. Thus, we only need to prove Corollary~\ref{cor: stab} for any fixed $m$ (and also allowing dependence on $m$). Also recall the bound in Theorem~\ref{th:unm_uniform_bd_stable} is uniform in both $m$ and $n$.

We also note we can further freely assume $\|\pi^{{n}}\|$ is small enough and such smallness can depend on this $m$.
Indeed, fix $\eps$, we can choose $m$ so that \eqref{eq: mgood} holds.  For  any given number $c_{m}$, we can reduce the proof of Corollary \ref{cor: stab} into the case $\|\pi^{{n}}\|\geq c_{m}$ and $\pi^{(n)}\leq c$. If $\|\pi^{{n}}\|\geq c_{m}$. We directly prove Corollary \ref{cor: stab} without reducing the $\un_{m}$. And what we are left is the case for stability of $\un_m$ with $\pi^{(n)}\leq c$.

We may only consider $m$ large, whose value may depend on $\eps$. Again for notational simplicity, we fix $\eps=\eps_{0}$ and $\rho$. We fix $n$ and $m$, and denote $\un_{m}$ by $v$, $\widetilde{u}^{(n)}_{m}$ by $\widetilde{v}$, and denote $t_{j}^{(n)}$ by $t_{j}$ as usual. Let
\begin{equation}
w := v - \widetilde{v} \quad \text{and} \quad U:= W - \widetilde{W}\;. 
\end{equation}
We will need a small constant $\eta$ similarly as in the previous section. Also recall $\Lambda_{ini}, \Lambda_{noi}$ are fixed throughout this article. Note that since $m$ is fixed, the nonlinearity essentially behaves linearly. 

We finally recall, since we a priori have $L_{\omega}^{\tilde{\rho}}\xX$ bound for all $\tilde{\rho}$, we are free to drop small probability sets. We need the following lemma. 

\begin{lem} \label{lem:iterate}
There is $h>0$, such that when $\delta_{0}$ is chosen small enough, one can always find  $\delta$,in \eqref{eq: enhanced}, small enough (depending on $\delta_{0},\rho,m, h$), such that if 
\begin{equation}
\|w\|_{L^{\rho}_{\omega}\chi(0,c)}< \delta_{0}
\end{equation}
and $d-c\sim h$,
then one has, 
\begin{equation}
\|w\|_{\chi(0,d)}\lesssim \delta _{0}
\end{equation}
\end{lem}

Once we have Lemma~\ref{lem:iterate}, we can then iterate it $\sim 1/h$ times and the desired stability follows. It should be expected $\delta \ll \delta_{0}$, and $-\log \delta_{0}\sim -\frac{1}{h} \log \eps$.

For technical convenience, we will only consider the case $c=t_{j}, d=t_{j'}$ for some $j<j'$. This does not cause problems since we assume $\|\pi\|$ is small enough, and such smallness could depend on $m$. Lemma~\ref{lem:iterate} will be reduced to the following bootstrap lemma. 

\begin{lem}\label{lem: onemorebootstrap}
Given the assumption of Lemma \ref{lem:iterate}, one can find $C>0$ such that if for $T \in [c,d]$, we have the bootstrap assumption
\begin{equation}\label{eq: onemoreas}
\|w\|_{L_{\omega}^{\rho}\xX(0,T)}\leq 2C \delta_{0},
\end{equation}
then one has 
\begin{equation}\label{eq: onemorebootstrap}
\|w\|_{L_\omega^{\rho}\xX(0,T) }\leq C \delta_{0}.
\end{equation}
The value of $C$ could depend on $m$, which is fixed in this section. 
\end{lem}

We first derive a few useful estimates before finally turning to the proof of Lemma~\ref{lem: onemorebootstrap} at the end of this section. 

\subsection{Equation for $w$ and a collection of estimates} \label{sub111}

We first write down the equation for $w$. For any $[a,t] \subset [c,d]$ (recalling \eqref{eq: fullexpansion}), we have
\begin{equation}\label{eq: perturbedduhamel}
\begin{aligned}
w(t)=&e^{i(t-a)\Delta}w(a) - i\int_{a}^{t} e^{i(t-s)\Delta} \Big(\phi_{m}(\|v\|_{\xX_{2}(0,s)}) \big(\nN(v)
-\nN(\widetilde{v}) \big) \Big) {\rm d}s\\
\pm &i \int_{a}^{t} e^{i(t-s)\Delta} \Big( \big(\phi_{m}(\|v\|_{\xX_{2}(0,s)})-\phi_{m}(\|\widetilde{v}\|_{\xX_{2}(0,s)}) \big) \nN(\widetilde{v}(s)) \Big) {\rm d}s\\
\pm&\int_{a}^{t} e^{i(t-s)\Delta} \left(\dws\int_{[s]}^{s} e^{i(s-r)\Delta} \left( \phi_{m}(\|v\|_{\xX_{2}(0,r)}) \nN(v) - \phi_{m}(\|\widetilde{v}\|_{\xX_{2}(0,r)})\nN(\widetilde{v}) \right) {\rm d}r \right) {\rm d}s\\
\pm&\int_{a}^{t} e^{i(t-s)\Delta} \left(\dus \int_{[s]}^{s} e^{i(s-r)\Delta} \phi_{m}(\|\widetilde{v}\|_{\xX_{2}(0,r)}) \nN(\widetilde{v}) {\rm d}r \right) {\rm d}s\\
-&i \int_{a}^{t} e^{i((t-s)} \left( \dws \cdot e^{i(s-[s]\Delta)} w([s]) \right) {\rm d}s\\
\pm &\int_{a}^{t} e^{i(t-s)}  \left( \dus \cdot e^{i(s-[s])\Delta} \widetilde{v}(s) \right) {\rm d}s\\
-&\int_{a}^{t} e^{(t-s)\Delta} \left[ \dws \int_{[s]}^{s} e^{i(s-r)\Delta} \left(\dws w((r)) {\rm d}r \right) \right] {\rm d}s\\
\pm&\int_{a}^{t} e^{(t-s)\Delta} \left[ \dus \int_{[s]}^{s} e^{i(s-r)\Delta} \left(\dws \cdot v(r) \right) {\rm d} r \right] {\rm d}s\\
\pm&\int_{a}^{t} e^{(t-s)\Delta} \left[\dws \int_{[s]}^{s} e^{i(s-r)\Delta} \left( \dus v(r) \right) {\rm d}r \right] {\rm d}s\;.
\end{aligned}
\end{equation}
Here, $\Delta_j U$ is defined similarly as $\Delta_j W$. We will treat all the terms with $\pm$ sign before them perturbatively, so the exact choices of the sign do not matter in the analysis. We introduce some notation to simplify the above expression. Let
\begin{itemize}
\item $S_{1}(a,t):=-i \int_{a}^{t} e^{i((t-s)} \left( \dws \cdot e^{i(s-[s]\Delta)} w([s]) \right) {\rm d}s$;
\item $S_{2}(a,t):=-\int_{a}^{t} e^{(t-s)\Delta} \left[ \dws \int_{[s]}^{s} e^{i(s-r)\Delta} \left(\dws w((r)) {\rm d}r \right) \right] {\rm d}s$;
\item $e_{1}(a,t):=\pm \int_{a}^{t} e^{i(t-s)\Delta} \left(\dws\int_{[s]}^{s} e^{i(s-r)\Delta} \left( \phi_{m}(\|v\|_{\xX_{2}(0,r)}) \nN(v) - \phi_{m}(\|\widetilde{v}\|_{\xX_{2}(0,r)})\nN(\widetilde{v}) \right) {\rm d}r \right) {\rm d}s$;
\item $e_{2}(a,t):= \pm i \int_{a}^{t} e^{i(t-s)\Delta} \Big( \big(\phi_{m}(\|v\|_{\xX_{2}(0,s)})-\phi_{m}(\|\widetilde{v}\|_{\xX_{2}(0,s)}) \big) \nN(\widetilde{v}(s)) \Big) {\rm d}s$;
\item $e_{3}(a,t):=\pm \int_{a}^{t} e^{i(t-s)\Delta} \left(\dus \int_{[s]}^{s} e^{i(s-r)\Delta} \phi_{m}(\|\widetilde{v}\|_{\xX_{2}(0,r)}) \nN(\widetilde{v}) {\rm d}r \right) {\rm d}s$;
\item $e_{4}(a,t):=\pm \int_{a}^{t} e^{i(t-s)}  \left( \dus \cdot e^{i(s-[s])\Delta} \widetilde{v}(s) \right) {\rm d}s$; 
\item $e_{5}(a,t):=\pm \int_{a}^{t} e^{(t-s)\Delta} \left[ \dus \int_{[s]}^{s} e^{i(s-r)\Delta} \left(\dws \cdot v(r) \right) {\rm d} r \right] {\rm d}s$; 
\item $ e_{6}(a,t):=\pm \int_{a}^{t} e^{(t-s)\Delta} \left[\dws \int_{[s]}^{s} e^{i(s-r)\Delta} \left( \dus v(r) \right) {\rm d}r \right] {\rm d}s$.
\end{itemize}
Now, we can write \eqref{eq: perturbedduhamel} as 
\begin{equation} \label{eq: pesimp}
\begin{split}
w(t) = &e^{i(t-a)\Delta}w(a) - i\int_{a}^{t} e^{i(t-s)\Delta} \Big(\phi_{m}(\|v\|_{\xX_{2}(0,s)}) \big(\nN(v)
-\nN(\widetilde{v}) \big) \Big) {\rm d}s\\
&+S_{1}(a,t) + S_{2}(a,t) +\sum_{i} e_{i}(a,t).
\end{split}
\end{equation}
We first present all the estimates for $S_{i}$ and $e_{j}$. Since the estimates are of same nature as in the proof of Theorem~\ref{th:unm_uniform_bd_stable}, we make a sketch of the arguments and highlight only the differences. We will work on time interval $[a,b] \subset [c,T] \subset [c,d]$. 

We start with $S_{1}$ and $S_{2}$. Similarly as \eqref{eq: estimatestragiht}, Lemma~\ref{le:source_max} (see also Remarks~\ref{remark: extrasmall1} and~\ref{remark: exs2}), we can find $S_{1}^{*}(t)$ and $S_{2}^{*}(t)$ such that the following lemma holds. 

\begin{lem}\label{lem: stabconsou}
	For $[c,d]$ such that $d-c \leq h$, we have
\begin{equation}
\|S_{i}(a,t)\|_{L_{x}^{10}} \lesssim S_{i}^{*}(t), \quad \|S_{i}^{*}(t)\|_{L_{\omega}^{\rho}L_{t}^{5}(c,T)} \lesssim (T-c)^{\gamma} \|w\|_{\xX(c,T)} \lesssim h^{\gamma} \|w\|_{L^{\rho}_{\omega}\xX(0,T)}.
\end{equation}
\end{lem}

We will not keep track of the exact value of $\gamma$, and in fact, we allow it to change line by line (but with a positive lower bound). We need a similar control for the $L_{x}^{2}$ norm for $S_{i}$. This is indeed easier since
\begin{equation}\label{eq: onemorestable}
\begin{aligned}
&\left\| \int_{a}^{t}e^{i(t-s)\Delta} \left( \dws \cdot e^{i(s-[s])\Delta} w([s]) \right) {\rm d}s \right\|_{L_{x}^{2}}\\
=&\left\| \int_{a}^{t} e^{i(-s)\Delta} \left( \dws \cdot e^{i(s-[s])\Delta} w([s]) \right) {\rm d}s \right\|_{L_{x}^{2}}\\
\lesssim &\sup_{\tau \in [c,d]} \left\| \int_{c}^{\tau}e^{-i s \Delta} \left( \dws \cdot e^{i(s-[s])\Delta} w([s]) \right) {\rm d}s \right\|_{L_{x}^{2}}\;.
\end{aligned}
\end{equation}
The last term does not depend on $t$. Similar observation works for $S_{2}$. Then, one may derive the following analogue of Lemma \ref{lem: stabconsou},
\begin{lem}\label{lem: stable}
There exists $\tilde{S}_{i}^{*}$, $i=1,2$ such that
\begin{equation}
\|S_{i}(a,t)\|_{L_{x}^{2}} \lesssim \tilde{S}_{i}^{*}, \qquad \|\tilde{S_{i}}^{*}\|_{L_{\omega}^{\rho}}\lesssim h^{\gamma}\|w\|_{L_{\omega}^{\rho}\xX(0,T)},
\end{equation}
The bounds are uniform in $t \in [0,T]$. 
\end{lem}

We finally collect the estimates for the $e_{i}$'s. All the bounds are allowed to depend on $m$. 

\begin{lem}\label{lem: estimatefore}
Let $[a,b]\subset [a,T] \subset [c,d]$. We have
\begin{itemize}
\item Estimate for $e_{1}$:
\begin{equation}\label{eq: 1e1}
\|e_{1}(a,\cdot)\|_{\xX(a,b)} \lesssim \big(\sup_{j}\| \Delta_{j}W \|_{L_{x}^{\infty}} \big) \cdot \|w\|_{\xX(0,b)}.
\end{equation}
\item Estimate for $e_{2}$:
\begin{equation}\label{eq: 1e2}
\|e_{2}(a,\cdot)\|_{\xX(a,b)} \lesssim \frac{1}{m} \|w\|_{\xX(0,b)} \|\widetilde{v}\|_{\xX_{2}(a,b)}^{5}.
\end{equation}
\item Estimate for $e_{3}$: 
\begin{equation}\label{eq:1e3}
\|e_{3}(a,\cdot)\|_{\xX(a,b)} \lesssim \sup_{j}\|\Delta_{j}U\|_{L_x^{\infty}}\;. 
\end{equation}

\item Estimate for $e_{4}, e_{5}, e_{6}$. There exists $e^{*}(t)$ and $\tilde{e}^{*}$ such that
\begin{equation}\label{eq:1e456}
\sum_{i=4}^{6}\|e_{i}(a,t)\|_{L_{x}^{10}} \leq e^{*}(t), \quad \|e^{*}(t)\|_{L^{\rho}_{\omega}L_{t}^{5}(0,1)}\lesssim \delta\;, 
\end{equation}
\begin{equation}\label{eq:2e456}
\sum_{i=4}^{6}\|e_{i}(a,t)\|_{L_{x}^{2}}\leq \tilde{e}^{*}, \quad \|\tilde{e}^{*}\|_{L_{\omega}^{\rho}}\lesssim \delta\;. 
\end{equation}
\end{itemize}
\end{lem}
\begin{proof}[Sketch proof of Lemma~\ref{lem: estimatefore}]

The estimate for $e_{1}(a,t)$ is similar to \eqref{eq: pe}. By Strichartz estimate \eqref{eq: stri}, we have
\begin{equation}\label{eq: aaa}
\begin{aligned}
&\|e_{1}(a,t)\|_{\xX(a,b)}\\
\lesssim&\left\| \dws\int_{[s]}^{s} e^{i(s-r)\Delta} \left(\phi_{m}(\|v\|_{\xX_{2}(0,r)}) \nN(v) - \phi_{m}(\|\tilde{v}\|_{\xX_{2}(0,r)}) \nN(\widetilde{v})\right) {\rm d}r \right\|_{L_{s}^{1}L_{x}^{2}[a,b]}\\
\lesssim &\sup_{j}\|\Delta_{j}W\|_{L_{x}^{\infty}} \cdot \left\|\phi_{m}(\|v\|_{\xX_{2}(0,r)}) \nN(v) - \phi_{m}(\|\tilde{v}\|_{\xX_{2}(0,r)}) \nN(\widetilde{v}) \right\|_{L_{t}^{1}L_{x}^{2}(a,b)}\;. 
\end{aligned}
\end{equation}
Thanks to the truncation $\phi_{m}$, one always has
\begin{equation}
\|\phi_{m}(\|v\|_{\xX_{2}(0,r)})\nN(v)\|_{L_{t}^{1}L_{x}^{2}}+\|\phi_{m}(\|\tilde{v}\|_{\xX_{2}(0,r)})\nN(\widetilde{v})\|_{L_{t}^{1}L_{x}^{2}}\leq C_{m}
\end{equation}
If $\|w\|_{\xX(0,T)}\geq 1$, then the desired estimate follows. Otherwise, if $\|w\|_{\xX(0,T)}\leq 1$, we further obtain
\begin{equation}
\|\phi_{m}(\|v\|_{\xX_{2}(0,r)})\widetilde{v}\|_{\xX_2(0,T)} + \|\phi_{m}(\|\tilde{v}\|_{\xX_{2}(0,r)})v\|_{\xX_2(0,T)}\| \leq C_{m}\;. 
\end{equation}
That means that the truncations for $v$ and $\widetilde{v}$ are essentially same. We split the term in concern as 
\begin{equation}
\begin{aligned}
& \|\phi_{m}(\|v\|_{\xX_{2}(0,r)}) \nN(v) - \phi_{m}(\|\tilde{v}\|_{\xX_{2}(0,r)}) \nN(\widetilde{v})\|_{L_{t}^{1}L_{x}^{2}(0,T)}\\
\leq &\|\phi_{m}(\|v\|_{\xX_{2}(0,r)}) (\nN(v)-\nN(\widetilde{v}))\|_{L_{t}^{1}L_{x}^{2}(0,T)} + \|(\phi_{m}(\|v\|_{\xX_{2}[0,r]}) - \phi_{m}(\|\widetilde{v}\|_{\xX_{2}[0,r]})) \nN(\widetilde{v})\|_{L_{t}^{1}L_{x}^{2}(0,T)}\;,
\end{aligned}
\end{equation}
and observe 
\begin{equation}
|\phi_{m}(\|v\|_{\xX_{2}[0,s]})-\phi_{m}(\|\widetilde{v}\|_{\xX_{2}[0,s]})|\lesssim \frac{1}{m} \|w\|_{\xX[0,s]},
\end{equation}
then estimate \eqref{eq: 1e1} follows.

The estimate for $e_{2}$ follows from Strichartz estimate and the observation
\begin{equation}
|\phi_{m}(\|v\|_{\xX_{2}[0,s]})-\phi_{m}(\|\tilde{v}\|_{\xX_{2}[0,s]})|\lesssim \frac{1}{m}\|w\|_{\xX[0,s]}.
\end{equation}
The estimate for $e_{3}$ is similar to that for $e_{1}$ except we replace $\Delta_{j} W$ by $\Delta_{j}U$, and we use the estimate
\begin{equation}
\|\phi_{m}(\|\widetilde{v}\|_{\xX_{2}[0,s]}) \nN(\widetilde{v}(s))\|_{L_{t}^{1}L_{x}^{2}} \lesssim C_{m}.
\end{equation}

For terms $e_{4}$, $e_{5}$, and $e_{6}$, we can treat similarly as in Lemmas~\ref{lem: stable} and~\ref{lem: stabconsou}, but here we do not need to restrict the analysis in small intervals\footnote{Such smallness is not useful here. In some sense, the extra smallness $h^{\gamma}$ is not small enough.}. However, we do observe in all those term, there is one $W$ been replaced by $U$, which gives a $\delta$ in the left hand due to \eqref{eq: enhanced}.
\end{proof}

\subsection{Proof of Lemma~\ref{lem: onemorebootstrap}}
Now we are ready to prove Lemma \ref{lem: onemorebootstrap}, which will conclude the proof of Corollary~\ref{cor: stab}. Note that once $h$ is fixed, we are allowed to take $\delta_{0}$ as small as we want (by choosing $\delta$ even smaller), but we will never use estimates of the form $\delta_{0}\lesssim h^{\alpha}$.

Recall we will need a small universal constant $\eta$. For almost surely every $\omega$, we will take a (random) partition of the interval $[c,T] $ into $c=a_{0}<a_{1}<...a_{J}=T$ such that for every $[a_{l}, a_{l+1}]$, the following properties for $v$ and $\widetilde{v}$ hold:
\begin{itemize}
\item Either $\|v\|_{\xX_{2}(a_{l},a_{l+1})} \leq \eta$ or $\phi_{m}(\|v\|_{\chi_{2}(0,s)}) = 0$ if $s \geq a_{l}$,
\item Either $\|\widetilde{v}\|_{\chi_{2}(a_{l},a_{l+1})} \leq \eta$ or $\phi_{m}(\|\widetilde{v}\|_{\xX_{2}(0,s)})=0$ if $s \geq a_{l}$. 
\end{itemize}
Thus, there can be most $J \sim m^{5}/\eta^{5} + 1 \leq C_{m}$ such intervals. In every such interval $[a_{l}, a_{l+1}]$, we use Strichartz and triangle inequalities to get
\begin{equation}
\|w\|_{\xX(a_{l},a_{l+1})} \lesssim \|w(a_{l})\|_{L_{x}^{2}} + \eta^{4} \|w\|_{\xX_{2}(a_{l}, a_{l+1})} + \sum_{i}\|S_{i}\|_{\xX(a,b)}\;. 
\end{equation}
Plug in the estimate for $e_{i}$'s and use the definition of $S_{i}^{*}(t), \tilde{S}_{i}^{*}$, (also note that the appearance of $e_{2}$ means $\|\widetilde{v}\|_{\xX_{2}(a_{l},a_{l+1})} \lesssim \eta$), we get
\begin{equation}
\begin{aligned}
\|w\|_{\xX(0,a_{l+1})} \lesssim  &\|w\|_{\xX(0,a_{l})} + \eta^{4} \|w\|_{\xX(0,a_{l+1})} + \sum_{i} \tilde{S}^{*}_{i} + \sum_{i} \|S^{*}_{i}\|_{L_{t}^{5}(c,T)}\\
+&C_{m} \sup_{j} \|\Delta_{j}W\|_{L_{x}^{\infty}} + \frac{1}{m} \|w\|_{\xX(0,a_{l+1})} \cdot \eta^{5} + C_{m} \sup_{j}\|\Delta_{j}U\|_{L_{x}^{\infty}}\\
+&\sum_{i=4}^{6} \|e^{*}(t)\|_{L_{t}^{5}(0,1)} + \tilde{e}^{*}_{i}\;. 
\end{aligned}
\end{equation}
Note that when $\eta$ is small enough, the two terms $\eta^{4} \|w\|_{\xX(0,a_{l+1})}$ and $\frac{1}{m} \|w\|_{\xX(0,a_{l+1})} \eta^{5}$ will be absorbed into the left side. Iterating the above formula $\sim C_{m}$ times (possibly with a different value of $C_m$), we get
\begin{equation} 
\begin{aligned}
\|w\|_{\xX(0,T)}&\leq C_{m} \|w\|_{\xX(0,c)} + C_{m} \bigg[ \Big(\sum_{i}\tilde{S}^{*}_{i} + \sum_{i}\|S^{*}_{i}\|_{L_{t}^{5}(c,T)} \Big)
+\sup_{j}\|\Delta_{j}W\|_{L_{x}^{\infty}}\\
&+ \sup_{j}\|\Delta_{j}U\|_{L_{x}^{\infty}} + \Big(\sum_{i=4}^{6}\|e^{*}(t)\|_{L_{t}^{5}(0,1)} +\tilde{e}^{*}_{i} \Big) \bigg]\;. 
\end{aligned}
\end{equation}
Taking $L_{\omega}^{\rho}$-norm on both sides, we obtain
\begin{equation}
\|w\|_{L^{\rho}_{\omega}\xX(0,T)} \leq C_m \left( \|w\|_{\xX(0,c)} + h^{\gamma} \|w\|_{L^{\rho}_{\omega}\xX(0,T)} + \delta + \|\pi^{(n)}\|^{\theta} \right).
\end{equation}
Here, we use the fact there exists $\theta>0$ such that $\|\sup_{j}\|\Delta_{j}W\|\|_{L_{\omega}^{\rho}L_{x}^{\infty}} \lesssim \|\pi^{(n)}\|^{\theta}$. One could choose $\theta=\frac{1}{2}-$. By choosing $h$ small enough according to $C_{m}$, and $\delta, \|\pi^{n}\|$ small enough according to $C_{m}$ and $\delta_{0}$, Lemma~\ref{lem: onemorebootstrap} follows.

\section{Proof of Proposition~\ref{pr:unm_to_um}}\label{sec:pr_WZ}

\subsection{Some reduction by Corollary~\ref{cor: stab}}

By Corollary \ref{cor: stab}, instead of Assumption~\ref{as:noise_initial}, we only need prove Proposition \ref{pr:unm_to_um} for smooth initial data $f\in L_{\omega}^{\infty}H^{1}_{x}$, and to study the noise which is finite dimensional (thanks to the summability condition \eqref{eq:summable}) and smooth in space. We can assume without loss of generality that
\begin{equation}\label{eq: furas}
f\in L_{\omega}^{\infty}H_{x}^{1}, \quad W(t,x) = B_1(t) V_1(x) + B_2(t) V_2(x)\;, 
\end{equation}
where $V_1$ and $V_2$ are Schwartz functions on $\RR$, since this already covers the treatment of both diagonal and cross terms, which will be of no difference to general finite dimensional case. The regularity assumptions on the initial data and noise give some H\"older regularity in time of the flow $u^{(n)}_{m}$ and $u^{m}$, which is essential to establish Wong-Zakai convergence.

Note that we deal with the truncated equation only (with $\phi_m$ in front of the nonlinearity, and all the bounds are allowed to depend on (the fixed) $m$. Again, we use $C_{m}$ to denote a generic constant depending on $m$ but its value may change from line to line. 

Let $w=u_{m}^{(n)}-u_{m}$. Similar to the previous section, we reduce the proof of Proposition \ref{pr:unm_to_um}, with the stronger assumption \eqref{eq: furas}, to the following claim. 

\begin{lem}\label{lem: boot1}
There exists $h>0$ such that when $\delta_{0}$ is small enough, one can always choose $\delta$ sufficiently small depending on $\delta_{0}$, $\rho$, $m$, $V_{1}$, $V_{2}$ and $f$ such that if $\|\pi^{n}\| \leq \delta$, and if one has 
\begin{equation}
\|w\|_{L^{\rho}_{\omega}\chi(0,c)}< \delta_{0}
\end{equation}
and $d-c \sim h$, then one has the bound
\begin{equation}
\|w\|_{\chi(0,d)}\lesssim \delta _{0}\;. 
\end{equation}
\end{lem}

Note that $w(0)=0$, and hence Proposition~\ref{pr:unm_to_um} follows from iterating the above lemma. Lemma~\ref{lem: boot1} in turn follows from the following bootstrap Lemma. 

\begin{lem}\label{lem: lll}
Given the assumptions of Lemma \ref{lem: boot1}, one can find $C > 0$ (allowed to depend on $m$) such that if one has the bootstrap assumption
\begin{equation}\label{eq: twomoreas}
\|w(t)\|_{L_{\omega}^{\rho}\xX[0,T]}\leq 2C\delta_{0}\;, \quad T \in [c,d]\;, 
\end{equation}
then one has the bootstrap estimate
\begin{equation}\label{eq: twomorebootstrap}
\|w(t)\|_{L_\omega^{\rho}\xX[0,T]}\leq C\delta_{0}.
\end{equation}
\end{lem}

It then remains to prove Lemma~\ref{lem: lll}, for which the rest of this section is devoted to.

\subsection{Well-posedness results in $H^1$}

Before we start proving Lemma~\ref{lem: lll}, we need some further well-posedness for $u_{m}$ and $u_{m}^{(n)}$ in high regularity space. We start with he following two propositions on the persistence of regularity. 

\begin{prop} \label{pr:unm_H1}
	With enhanced assumption \eqref{eq: furas}, there exists $C=C \big(m, \rho, \Lambda_\ini, \Lambda_\noi', \|f\|_{L_{\omega}^{\rho}H_x^1} \big)$ such that
	\begin{equation} \label{eq:unm_H1}
	\sup_{n} \|u^{(n)}_m\|_{L_{\omega}^{\rho}L_{t}^{\infty}H_x^1} \leq C\;, \qquad \|u_m\|_{L_{\omega}^{\rho} L_{t}^{\infty} H_x^1} \leq C. 
	\end{equation}
\end{prop}
\begin{proof}
	The bound for $u_m$ is proved in \cite[Proposition~3.2]{snls_mass_critical}. The uniform-in-$n$ bound for $\{u^{(n)}_{m}\}$ can be proved by following exactly the same procedure as that for $u_m$ and by further local Duhamel expansion as in Section~\ref{sec:unm_bd}. We omit the details here. 
\end{proof}

\begin{prop} \label{pr:unm_H1_time_cont}
	Let $\alpha \in (0,1)$. For every $m$ and $\rho$, there exists $C=C \big( m, \rho,\Lambda_\ini, \Lambda_{\noi}'\big)$ such that
	\begin{equation}
	\|u^{(n)}_{m}(t) - u^{(n)}_m(s)\|_{L_{\omega}^{\rho}H_x^\alpha} \leq C |t-s|^{\frac{1-\alpha}{2}} \Big( \|u^{(n)}_{m}\|_{L_{\omega}^{\rho}L_{t}^{\infty} H_{x}^{1}} + \|u^{(n)}_{m}\|_{L_{\omega}^{5\rho}L_{t}^{\infty}H_{x}^{1}}^{5} \Big). 
	\end{equation}
	As a consequence, we have
	\begin{equation}
	\|u^{(n)}_{m}\|_{L_{\omega}^{\rho} \cC_{t}^{\beta} H_{x}^{\alpha}} \leq C \Big( \|u^{(n)}_{m}\|_{L_{\omega}^{\rho}L_{t}^{\infty} H_{x}^{1}} + \|u^{(n)}_{m}\|_{L_{\omega}^{5\rho}L_{t}^{\infty}H_{x}^{1}}^{5} \Big) \leq C
	\end{equation}
	for every $\beta < \frac{1-\alpha}{2}$. In the second claim, the constant $C$ also depends on $\beta$. The same bounds are true for $u_m$. 
\end{prop}
\begin{proof}
	Again, for simplicity, we prove the bounds for $u_m$ only. For every $s<t$, we have
	\begin{equation}
	\begin{split}
	u_m(t) - u_m(s) = &\big(e^{i(t-s)\Delta} u_m(s) - u_m(s)\big) - i \int_{s}^{t} e^{i(s-r)\Delta} \Big( \phi_m \big( \|u_m\|_{\xX_2(0,r)} \big) \nN \big( u_m(r) \big) \Big) {\rm d}r\\
	&- i \int_{s}^{t} e^{i(t-r)\Delta} \big( u_m(r) {\rm d} W(s) \big) - \frac{1}{2} \sum_{k} \int_{s}^{t} e^{i(t-r)\Delta} \big( V_k^2 u_m(r) \big) {\rm d}r. 
	\end{split}
	\end{equation}
	The last three terms above can be controlled directly via dispersive estimates and in the stochastic integral case, also with Burkholder inequality. To bound the first term, one really uses the flow being in $H^1$, so that
	\begin{equation}
	\|e^{i(t-s)\Delta}u_m(s) - u_m(s)\|_{H_x^\alpha} \leq C |t-s|^{\frac{1-\alpha}{2}} \|u_m(s)\|_{H_x^1}. 
	\end{equation}
	Hence, one get the desired bound on $\|u_m(t) - u_m(s)\|_{L_{\omega}^{\rho}H_x^\alpha}$. The bound for $\|u_m\|_{L_{\omega}^{\rho}\cC_t^\beta H_x^\alpha}$ follows from Kolmogorov criteria. 
\end{proof}

\subsection{Equation for $u_{m}$ and $w=u_{m}^{(n)}-u_{m}$, and a collection of estimates}

We rewrite the equation for $u_{m}$ so that it will be suitable to do comparison of to $u_{m}^{n}$. Again, we short $t_j$ for $t_j^{(n)}$, and let $[s]$ and $j(s)$ be the same as the previous sections. We denote $u_{n}^{m}$ by $v$ and denote $u_{m}$ by $\widetilde{v}$.  We will present the arguments with consistent notations in the previous sections. 

Note that $v$ solves \eqref{eq: fullexpansion} except that we need to use $\phi_{m}(\|v\|_{\xX_{2}(0,\tau)})(\nN(v(\tau)))$ to replace $\nN(v(\tau))$.

We recall \eqref{eq:um_Duhamel}, and rewrite the equation of $\widetilde{v} = u_m$ as a perturbation of the equation which $v=u^{(n)}_{m}$ satisfies. We have
\begin{equation}\label{eq:mimicumn}
\begin{aligned}
\widetilde{v}(t) = &e^{it\Delta} \widetilde{v}(a) - i \int_{a}^{t} e^{i(t-s)\Delta} \left(\phi_{m}(\|\tv\|_{\xX_{2}(0,s)}) \nN(\tv(s)) \right) {\rm d}s\\
&-i \int_{a}^{t} e^{i(t-s)\Delta} \left(\dws \cdot e^{i(s-[s])\Delta}\tv([s]) \right) {\rm d}s\\
&- \int_{a}^{t} e^{i(t-s)\Delta} \left[ \dws \int_{[s]}^{s}e^{i(s-r)\Delta} \left( \frac{\Delta_{j(s)}W}{t_{j(s)+1}-t_{j(s)}}\tv(r) \right) {\rm d}r \right] {\rm d}s\\
&\pm i \int_{t_{j(a)+1}}^{[t]} e^{i(t-s)\Delta} \left( \left( \tv(s)-\tv([s]) \right) {\rm d} W(s) \right)\\
&\pm \frac{1}{2} \int_{t_{j(a)+1}}^{[t]} e^{i(t-s)\Delta} \left[ \V^{2}\tv(s)- \frac{2 \Delta_{j(s)}W}{t_{j(s)+1}-t_{j(s)}}\int_{[s]}^{s} e^{i(s-r)\Delta} \left( \frac{\Delta_{j(s)}W}{t_{j(s)+1}-t_{j(s)}}\tv(r) \right) {\rm d}r \right] {\rm d}s\\
&\pm i\int_{a}^{t_{j(a)+1}}e^{i(t-s)\Delta} \left(\tv(s) {\rm d}W(s) \right) \pm \int_{[t]}^{t} e^{i(t-s)\Delta} \left( \tv(s) {\rm d} W(s) \right)\\
&\pm \frac{1}{2}\int_{a}^{t_{j(a)+1}} e^{i(t-s)\Delta} \left( \V^{2} \tv(s) \right) {\rm d}s +\pm \frac{1}{2} \int_{[t]}^{t} e^{i(t-s)\Delta} \left( \V^{2} \tv(s) \right) {\rm d}s\;. 
\end{aligned}
\end{equation}
Note that the above is the expansion of $\widetilde{v} = u_m$ (which satisfies \eqref{eq:um_Duhamel}) in terms of the Duhamel formula for $v = u_m^{(n)}$ in \eqref{eq: fullexpansion} with some additional error terms. The error terms are those with the $\pm$ signs above. They will be treated in a purely perturbative way, and the actual signs will not matter. We use $\V^{2}$ to denote $V_{1}^{2}+V_{2}^{2}$.

Recall that $w:=v-\tilde{v}=u_{m}^{(n)}-u_{m}$, we have
\begin{equation}\label{eq:finallongeq}
\begin{aligned}
w(t) &= e^{it\Delta} w(a) - i \int_{a}^{t} e^{i(t-s)\Delta} \left(\phi_{m}(\|v\|_{\xX_{2}(0,s)})\nN(v(s)) - \phi_{m}(\|\tv\|_{\xX_{2}(0,s)}) \nN(\tv(s)) \right) {\rm d}s\\
&-i \int_{a}^{t} e^{i(t-s)\Delta} \left( \dws \cdot e^{i(s-[s])\Delta} w([s]) \right) {\rm d}s\\
&- \int_{a}^{t} e^{i(t-s)\Delta} \left[ \dws \int_{[s]}^{s} e^{i(s-r)\Delta} \left( \dws \cdot w(r) \right) {\rm d}r \right] {\rm d}s\\
&\pm i \int_{t_{j(a)+1}}^{[t]} e^{i(t-s)\Delta} \left( (\tv(s)-\tv([s])) {\rm d} W(s) \right)\\
&\pm \frac{1}{2} \int_{t_{j(a)+1}}^{[t]} e^{i(t-s)\Delta} \left[ \V^{2}\tv(s)- \frac{2 \Delta_{j(s)} W}{t_{j(s)+1}-t_{j(s)}} \int_{[s]}^{s} e^{i(s-r)\Delta}\left( \frac{\Delta_{j(s)}W}{t_{j(s)+1}-t_{j(s)}}\tv(r)\right) {\rm d}r \right] {\rm d} s\\
&\pm i \int_{a}^{t_{j(a)+1}} e^{i(t-s)\Delta} \left( w(s) {\rm d} W(s) \right) \pm \int_{[t]}^{t} e^{i(t-s)\Delta} \left( w(s) {\rm d} W(s) \right)\\
&\pm \frac{1}{2} \int_{a}^{t_{j(a)+1}} e^{i(t-s)\Delta} \left( \V^{2} \tv(s) \right) {\rm d}s \pm \frac{1}{2} \int_{[t]}^{t} e^{i(t-s)\Delta} \left(\V^{2} \tv(s) \right) {\rm d}s\\
&\pm \int_{a}^{t} e^{i(t-s)\Delta} \left(\dws \int_{[s]}^{s} e^{i(s-r)\Delta} \left( \phi_{m}(\|v\|_{\xX_{2}(0,s)}) \nN(v(r)) \right) {\rm d}r \right) {\rm d}s\;. 
\end{aligned}
\end{equation}
We again introduce some notations to simplify the above expression. Let
\begin{itemize}
\item $M_{1}(a,t) = -i \int_{a}^{t} e^{i(t-s)\Delta} \left( \dws \cdot e^{i(s-[s])\Delta} w([s]) \right) {\rm d}s$,
\item $M_{2}(a,t) = - \int_{a}^{t} e^{i(t-s)\Delta} \left[ \dws \int_{[s]}^{s} e^{i(s-r)\Delta} \left( \dws \cdot w(r) \right) {\rm d}r \right] {\rm d}s$, 
\item $M_{3}(a,t) = \pm i \int_{t_{j(a)+1}}^{[t]} e^{i(t-s)\Delta} \left( (\tv(s)-\tv([s])) {\rm d} W(s) \right)$, 
\item $g_{1}(a,t) = \pm \frac{1}{2} \int_{t_{j(a)+1}}^{[t]} e^{i(t-s)\Delta} \left[ \V^{2}\tv(s)- \frac{2 \Delta_{j(s)} W}{t_{j(s)+1}-t_{j(s)}} \int_{[s]}^{s} e^{i(s-r)\Delta}\left( \frac{\Delta_{j(s)}W}{t_{j(s)+1}-t_{j(s)}}\tv(r)\right) {\rm d}r \right] {\rm d} s$, 
\item $g_{2}(a,t) = \pm i \int_{a}^{t_{j(a)+1}} e^{i(t-s)\Delta} \left( w(s) {\rm d} W(s) \right) \pm \int_{[t]}^{t} e^{i(t-s)\Delta} \left( w(s) {\rm d} W(s) \right)$, 
\item $g_{3}(a,t) = \pm \frac{1}{2} \int_{a}^{t_{j(a)+1}} e^{i(t-s)\Delta} \left( \V^{2} \tv(s) \right) {\rm d}s \pm \frac{1}{2} \int_{[t]}^{t} e^{i(t-s)\Delta} \left(\V^{2} \tv(s) \right) {\rm d}s$, 
\item $g_{4}(a,t) = \pm \int_{a}^{t} e^{i(t-s)\Delta} \left(\dws \int_{[s]}^{s} e^{i(s-r)\Delta} \left( \phi_{m}(\|v\|_{\xX_{2}(0,s)}) \nN(v(r)) \right) {\rm d}r \right) {\rm d}s$. 
\end{itemize}
Now, we may rewrite the equation for $w$ as 
\begin{equation}
\begin{aligned}
w(t) = &e^{it\Delta}w(a ) - i \int_{a}^{t} e^{i(t-s)\Delta} \left(\phi_{m}(\|v\|_{\xX_{2}(0,s)})\nN(v(s)) - \phi_{m}(\|\tv\|_{\xX_{2}(0,s)}) \nN(\tv(s)) \right) {\rm d}s\\
&+\sum_{i} M_{i}(a,t) + \sum_{i} g_{i}(a,t)\;. 
\end{aligned}
\end{equation}
We first summarize the estimates for $M_i$ and $g_i$ in the following lemma, which is analogues to Lemma \ref{lem: stabconsou}, ~\ref{lem: stable}, and~\ref{lem: estimatefore}. 

\begin{lem}\label{lem:esfornewe}
Let $[a,b]\subset [a,T] \subset [c,d]$ and $d-c\sim h$. We have the following: 
\begin{itemize}
\item Estimates for  $M_{1}$, $M_{2}$ and $M_{3}$. There exist $\theta>0$ and  $M_{i}^{*}(t),\tilde{M}_{i}^{*}$ such that
\begin{equation}
\|S_{i}(a,t)\|_{L_{x}^{10}} \lesssim M_{i}^{*}(t), \quad \|S_{i}(a,t)\|_{L_{x}^{2}} \lesssim \tilde{M}_{i}^{*},
\end{equation}
uniformly in $t \in [a,T]$ and for all $i=1,2,3$, and
\begin{equation}
\begin{aligned}
&\|M_{i}^{*}(t)\|_{L_{\omega^{\rho}}L_{t}^{5}(c,T)} \lesssim h^{\theta}\|w\|_{L^{\rho}_{\omega}\xX(0,T)}, \quad i=1,2.\\
&\|M_{3}^{*}(t)\|_{L_{\omega}^{\rho}L_{t}^{5}[c,T]} + \|\tilde{M}_{3}\|_{L^{\rho}_{\omega}} \lesssim \|\pi^{(n)}\|^{\theta}.
\end{aligned}
\end{equation}
\item Estimates for $g_{1}$ and $g_{2}$. There exist $g_{i}^{*}(t)$ and $\tilde{g}_{i}^{*}$ for $i=1,2$ such that
\begin{equation}
\|g_{i}(a,t)\|_{L_{x}^{10}} \lesssim g_{i}^{*}(t), \quad \|g_{i}(a,t)\|_{L_{x}^{2}} \lesssim \tilde{g}_{i}^{*}, \quad t \in [c,T]\;, 
\end{equation}
and
\begin{equation}
\|g_{i}^{*}\|_{L_{\omega}^{\rho}L_{t}^{5}(c,T)}+\|\tilde{g}_{i}^{*}\|_{L^{\rho}_{\omega}} \lesssim \|\pi^{(n)}\|^{\theta}
\end{equation}
\item Estimates for $g_{3}$. 
\begin{equation}
\|g_{3}(a,t)\|_{L^{2}_{x} \cap L_{x}^{10}} \lesssim \|\pi^{(n)}\|^{\theta}, \quad t \in [c,T].
\end{equation}
\item Estimate for $g_{4}$
\begin{equation}
\|g_{4}(a,t)\|_{\xX(a,b)} \lesssim \sup_{j} \|\Delta_{j}W\|_{L_{x}^{\infty}}\;. 
\end{equation}
\end{itemize}
All the bounds are allowed to depend on $m$, but are uniform in $t$. 
\end{lem}

\begin{remark}
The exact value of $\theta$ may change from line to line, but we only need such $\theta$ be positive. The point is that when $h$ and $\|\pi^{(n)}\|$ are small enough (depending on $m$), the terms $h^{\theta}$ and $\|\pi^{(n)}\|^{\theta}$ could always overcome the loss of any large constant (depending on $m$). 
\end{remark}

\begin{proof}[Proof of Lemma \ref{lem:esfornewe}]
Since the structure of the proof is very similar to Lemma \ref{lem: estimatefore}, we will sketch for he  parts which are similar and only  focus  highlight the difference. The main difference will be in the estimates for $M_{3}$ $g_{1}$ and $g_{2}$, which rely on the H\"older continuity of the flow. $g_{1}$ turns out to be the most technically complicated term. 
  
The estimates for $M_{1}$ and $M_{2}$ are completely the same as those for $S_{1}$ and $S_{2}$. 

For $M_3$, one may take $M_{3}^{*}
(t):=\sup_{c\leq \tau\leq T}\|\int_{c}^{\tau}e^{i(t-s)}(\tilde{v}(s))-\tilde{v}([s])\|_{L_{x}^{10}}$, and $\tilde{M}_{3}^{*} := \sup_{c \leq \tau \leq T} \|\int_{c}^{\tau} e^{i(-s)} (\widetilde{v}(s))-\widetilde{v}([s])\|_{L_{x}^{2}}$. Arguing similarly as the bound for $S_{1}$, we get
\begin{equation}\label{eq: almost}
\|M_{3}^{*}(t)\|_{L_{\omega}^{\rho} L_{t}^{5}} + \|\tilde{M}_{3}\|_{L_{\omega}^{\rho}} \lesssim \|\tv(s)-\tv([s])\|_{L_{\omega}^{\rho}L_{x}^{2}}\;, 
\end{equation}
where we have enlarged the factor $h^{\theta}$ to order $1$. Using the H\"older continuity of the flow in Proposition~\ref{pr:unm_H1_time_cont}, and observe that $|s-[s]|\leq \|\pi^{n}\|$, we obtain
\begin{equation}
\|M_{3}^{*}(t)\|_{L_{\omega}^{\rho}L_{t}^{5}} + \|\tilde{M}^{3}\|_{L_{\omega}^{\rho}} \lesssim \|\pi^{(n)}\|^{\theta} \quad \text{for some} \phantom{1} \theta>0.
\end{equation}
Now we turn to the estimate for $g_{1}$. We spit it as
\begin{equation}
g_{1}(a,t) = g_{11}(a,t) + g_{12}(a,t)\;, 
\end{equation}
where
\begin{equation}
g_{11}(a,t) = \pm \frac{1}{2} \int_{t_{j(a)}+1}^{[t]}e^{i(t-s)\Delta} \left[ \V^{2} \tv([s]) - \frac{2\Delta_{j(s)}W}{t_{j(s)+1}-t_{j(s)}} \int_{[s]}^{s} \frac{\Delta_{j(s)}W}{t_{j(s)+1}-t_{j(s)}} \tv([s]) {\rm d}r \right] {\rm d}s\;,
\end{equation}
and
\begin{equation}
\begin{aligned}
g_{12}(a,t)=&\pm \frac{1}{2} \int_{t_{j(a)}+1}^{[t]} e^{i(t-s)\Delta} \bigg[ \V^{2} \big( \tv(s)-\tv([s]) \big)\\
&-\frac{2\Delta_{j(s)}W}{t_{j(s)+1}-t_{j(s)}} \int_{[s]}^{s} \left[ e^{i(s-r)\Delta} \left( \dws \cdot \tv(r) \right)  - \frac{\Delta_{j(s)}W}{t_{j(s)+1}-t_{j(s)}}\tv([s]) \right] {\rm d}r \bigg] {\rm d}s\;. 
\end{aligned}
\end{equation}
We only need to find the associated $g^{*}_{1i}(t)$ and $\tilde{g}^{*}_{1i} $ for $g_{1i}(a,t)$, $i=1,2$. Recall we are working on the simple noise $W(t,x) = \sum_i B_i(t) V_i(x)$ and $V_i$ are Schwartz functions. 

We first give the estimate for $g_{12}$. Recall from Proposition~\ref{pr:unm_H1_time_cont} that $\widetilde{v}$ is H\"older continuous with the estimate
\begin{equation}\label{eq: holderwork}
\|\widetilde{v}\|_{L^{2\rho}_{\omega} \cC_{t}^{\theta} L_{x}^{2}(0,1)} \leq C_{m},
\end{equation}
which also implies $\|\widetilde{v}\|_{L^{\rho}_{\omega} \cC_{t}^{\theta} L_{x}^{2}(0,1)} \leq C_{m}$.

Observe that $\|\tv(s)-\tv([s])\|_{L_{x}^{2}} \leq \|\pi^{(n)}\|^{\theta} \|\widetilde{v}\|_{\cC^{\theta}_{t}L_{x}^{2}}$, and $\|\tv([s])-\tv(r)\|_{L_{x}^{2}} \leq |t_{j(s)+1}-t_{j(s)}|^{\theta} \|\pi^{(n)}\|^{\theta} \|\widetilde{v}\|_{\cC^{\theta}_{t}L_{x}^{2}}$ for all $r \in ([s],s)$. We may use the pointwise estimates
\begin{equation}\label{eq: pointe}
\begin{aligned}
&\|e^{i(t-s)\Delta} \big(\V^2 (\tv(s)-\tv([s]))\big)\|_{L_{x}^{2} \cap L_{x}^{10}} \lesssim (t-s)^{-2/5} \|\pi^{(n)}\|^{\theta}\|\widetilde{v}\|_{C^{\theta}_{t}L_{x}^{2}}\;\\
&\left\| e^{i(t-s)\Delta} \left[ \frac{2\Delta_{j(s)}W}{t_{j(s)+1}-t_{j(s)}} \int_{[s]}^{s} \left[ e^{i(s-r)\Delta} \left( \dws \tv(r) \right) - \dws \tv([s]) \right] {\rm d}r \right] \right\|_{L_{x}^{2} \cap L_{x}^{10}}\\
&\lesssim (t-s)^{-\frac{2}{5}} \frac{\sum_{i=1}^{2}|\Delta_{j(s)}B_i|^{2}}{t_{j(s)+1}-t_{j(s)}} \left( t_{j(s)+1}-t_{j(s)} \right)^{\theta} \|\widetilde{v}\|_{\cC^{\theta}_{t}L_{x}^{2}}\\
&\lesssim (t-s)^{-\frac{2}{5}} \|\pi^{(n)}\|^{\frac{\theta}{2}} \; \sup_{j} \frac{ \sum_{i=1}^{2} |\Delta_j B_i|^{2}}{(t_{j+1}-t_{j})^{1-\frac{\theta}{2}}} \cdot \|\widetilde{v}\|_{\cC^{\theta}_{t}L_{x}^{2}}\;. 
\end{aligned}
\end{equation}
where we have used the dispersive estimate  $\|e^{i(t-s)\Delta}\|_{L_{x}^{9/10}\rightarrow L_{x}^{10}}\lesssim (t-s)^{-2/5}$, the unitarity of $e^{it\Delta}$ in $L_{x}^{2}$, and the estimate for the difference between $e^{i(s-r)\Delta}(V(x)\tilde{v}(r))$ and $V(x)\tilde{v}([s])$ in Lemmas~\ref{le:semi_group_cont} and~\ref{le:op_change_order}. Thus, we get
\begin{equation}\label{eq: 333}
\|g_{12}(a,t)\|_{L_{x}^{2}\cap L_{x}^{10}} \lesssim 
\left( \int_{0}^{1}(t-s)^{-\frac{2}{5}} {\rm d}s \right) \left\{\|\pi^{(n)}\|^{\frac{\theta}{2}} \sup_{j}\frac{\sum_{i=1}^{2} |\Delta_j B_i|^{2}}{(t_{j+1}-t_{j})^{1-\frac{\theta}{2}}} \|\widetilde{v}\|_{\cC^{\theta}_{t}L_{x}^{2}}+\|\pi^{(n)}\|^{\theta}\|\tilde{v}\|_{C^{\theta}_{t}L_{x}^{2}}\right\}\;. 
\end{equation}
We can take $g_{12}^{*}(t)$ and $\tilde{g}_{12}$ both to be right hand side of \eqref{eq: 333}, the desired estimate then follows from \eqref{eq: holderwork} and the observation that $\sup_{j}\|\frac{|\Delta_j B_i|^{2}}{(t_{j+1}-t_{j})^{1-\theta/2}}\|_{L^{2\rho}}\lesssim 1$. 

We now turn to for $g_{11}$. Observe
\begin{equation}
g_{11}(a,t) = \sum_{[t_{j},t_{j+1}]\subset [a,t]}\int_{t_{j}}^{t_{j+1}} e^{i(t-s)\Delta} \left( \V^{2} \tv(t_{j}) - \sum_{i,i'=1,2} \frac{\Delta_j B_i \Delta_j B_{i'}}{(t_{j+1}-t_{j})^{2}} V_{i}V_{i'} \int_{t_{j}}^{s} \widetilde{v}([s]) {\rm d}r \right) {\rm d}s;. 
\end{equation}
We further split it as 
\begin{equation}
g_{11}(a,t)=g_{111}(a,t)+g_{112}(a,t)
\end{equation}
where
\begin{equation}
g_{111}(a,t) = \sum_{[t_{j},t_{j+1}]\subset [a,t]}\int_{t_{j}}^{t_{j+1}} e^{i(t-s)\Delta} \left( \V^{2} \tv(t_{j}) - \sum_{i,i'=1,2} \frac{\Delta_j B_i \Delta_j B_{i'}}{(t_{j+1}-t_{j})^{2}} V_{i}V_{i'} \int_{t_{j}}^{s} \widetilde{v}([s]) {\rm d}r \right) {\rm d}s\;, 
\end{equation}
and $g_{112}=g_{11}-g_{111}$.

The idea is because of Lemma~\ref{le:semi_group_cont}, we can trade some regularity to replace the $e^{i(t-s)\Delta}$ in $[t_{j},t_{j+1}]$ by $e^{i(t-t_{j})}$. The estimate of $g_{112}$ is similar to $g_{12}$, and we omit the details. We focus on the estimate of $g_{111}$, and construct the associated $g_{111}^{*}(t)$ and $\tilde{g}_{111}^{*}$. Observe that
\begin{equation}
g_{111}(a,t) = \pm \frac{1}{2} \sum_{[t_{j},t_{j+1}] \subset [a,t]} \left[ (t_{j+1}-t_{j}) e^{i(t-t_{j})\Delta} \sum_{i,i'=1,2} V_{i} V_{i'} \left(\delta_{ii'} - \frac{\Delta_j B_i \Delta_j B_{i'}}{t_{j+1}-t_{j}}\right) \right]\;, 
\end{equation}
where $\delta_{ii'}=1$ for $i=i'$, and $0$ otherwise. For notation convenience, we short $\delta_{ii'}-\frac{\Delta_j B_i \Delta_j B_{i'}}{(t_{j+1}-t_{j})}$ as $c_{ii'}^{(j)}$. Let
\begin{equation}
\begin{aligned}
g_{111}^{*}(t) &:= \sup_{j_{0}} \left\| \sum_{j\leq j_{0}} (t_{j+1}-t_{j}) e^{i(t-t_{j})\Delta} \sum_{i,i'} V_{i}V_{i'} c_{ii'}^{(j)} \right\|_{L_{x}^{10}}\\
g_{111}^{*} &:= \sup_{j_{0}} \left\| \sum_{j\leq j_{0}} (t_{j+1}-t_{j}) e^{i(-t_{j})\Delta}\sum_{i,i'} V_{i}V_{i'} c_{ii'}^{(j)} \right\|_{L_{x}^{2}}\;. 
\end{aligned}
\end{equation}
The key observation is that $c^{(j)}_{ii'}$ has mean $0$. Thus, the two quantities in inside the norms above are discrete martingales (in $j$) in $L_x^{10}$ and $L_x^2$ respectively (fixing $t$). 

The remaining arguments are similar to that for $\sm$ in Lemma~\ref{le:source_max}. We sketch the estimate for $g_{111}^{*}(t)$ but skip that of $\tilde{g}_{111}^{*}$.

Again, by \eqref{eq: min}, we may fix $t$, and just estimate
$\|g_{111}^{*}(t)\|_{L^{\rho}_{\omega}L_{x}^{10}}$. It follows from the Burkholder inequality \eqref{eq: burkholder} that
\begin{equation}\label{eq: jiayou}
\|g_{111}^{*}(t)\|_{L^{\rho}_{\omega}L_{x}^{10}}^{\rho} \lesssim \left\|\sum_{j} (t-t_{j})^{-\frac{4}{5}} (t_{j+1}-t_{j})^{2} \left( \sum_{i,i'}c_{ii'}^{(j)} \right)^{2} \right\|_{L_{\omega}^{\rho/2}}^{\rho/2}. 
\end{equation}
Also note that $c_{i'i}^{(j)}$ are independent for different $j$. Observe
\begin{equation}
\left\| \sum_{j} (t-t_{j})^{-\frac{4}{5}} (t_{j+1}-t_{j})^{2} \left(\sum_{i,i'} c_{ii'}^{(j)} \right)^{2} \right\|_{L_{\omega}^{\rho/2}} \lesssim \sum_{j} (t-t_{j})^{-\frac{4}{5}} (t_{j+1}-t_{j})^{2} \lesssim \|\pi^{(n)}\|\;. 
\end{equation}
The estimate for $g_{111}^{*}$ now follows, and we can conclude the bound for $g_{1}(a,t)$. 

The estimate for $g_{2}(a,t)$ is exactly the same as that for $M_{3}$, and we omit the details. 

We now turn to $g_{3}(a,t)$. One can directly apply the triangle inequality and dispersive estimate \eqref{eq: dispersive} to derive
\begin{equation}
\|g_{3}(a,t)\|_{L_{x}^{2}\cap L_{x}^{10}}\lesssim \int_{[t_{j(a)+1},a]\cup [[t],t]]}|t-s|^{-\frac{2}{5}}ds \lesssim \|\pi^{(n)}\|^{\frac{1}{5}}\;, 
\end{equation}
where we have used $t_{j(a)+1}-a\leq \|\pi^{n}\|$ and $t-[t]\leq \|\pi^{n}\|$. 

We finally turn to $g_{4}(a,t)$. This estimate is similar to \eqref{eq: redobystri},  \eqref{eq: pe} as well as the term $e_{3}$ in Lemma~\ref{lem: estimatefore}. Observe
\begin{equation}
\|\phi_{m}(\tilde{v})\nN(v)\|_{L_{t}^{1}L_{x}^{2}}\leq C_{m},
\end{equation} 
and the desired estimate follows. We have thus completed the proof for all the bounds in Lemma~\ref{lem:esfornewe}. 
\end{proof}

\subsection{Proof of Lemma \ref{lem: lll}}

We are now ready to prove Lemma~\ref{lem: lll}, which will then conclude this section. The arguments are very similar to those for Lemma~\ref{lem: onemorebootstrap}. 

Let $\eta>0$ be a small parameter (independent of $n$). For almost surely every $\omega$, we take a (random) partition of the interval $[c,T] $ into $c=a_{0}<a_{1}<...a_{J}=T$ such that for every $[a_{l}, a_{l+1}]$, the following properties hold: 
\begin{itemize}
\item For $v$,  either $\|v\|_{\xX_{2}[a_{l},a_{l+1}]}\leq \eta$ or $\phi_{m}(\|v\|_{\chi_{2}(0,s)})=0$ for $s\geq a_{l}$,
\item For $\tv$, either $\|\widetilde{v}\|_{\chi_{2}[a_{l},a_{l+1}]} \leq \eta$ or $\phi_{m}(\|\tv\|_{\xX_{2}(0,s)})=0$ for $s\geq a_{l}$. 
\end{itemize}
Thus, there can be most $J\sim m^{5}/\eta^{5}+1\leq C_{m}$ such  intervals. Observe that in every $[a_{l}, a_{l+1}]$, one has 
\begin{equation}
\left\| \phi_{m}(\|v\|_{\xX_{2}(0,s)}) \nN(v) - \phi_{m}(\|\tv\|_{\xX_{2}(0,s)}) \nN(\tv) \right\|_{L_{t}^{1}L_{x}^{2}(a_{l},a_{l+1})} \lesssim \eta^{4}\;. 
\end{equation}
To see why the above bound is true, one just needs to consider the following cases separately: 
\begin{itemize}
\item Both $\phi_{m}(\|v\|_{\xX_{2}(0,s)})$ and $ \phi_{m}(\|\tv\|_{\xX_{2}(0,s)})$ are $0$ for $s\geq a_{l}$. Then there is nothing to prove. 
\item Neither $\phi_{m}(\|v\|_{\xX_{2}(0,s)})$ nor $ \phi_{m}(\|\tv\|_{\xX_{2}(0,s)})$ is $0$, then one has $\|v\|_{\xX_{2}(a_{l},a_{l+1})} \leq \eta$ and $\|\widetilde{v}\|_{\xX_{2}(a_{l},a_{l+1})}\leq \eta$. We then split the whole term as
\begin{equation}
\phi_{m}(\|v\|_{\xX_{2}(0,s)}) \left( \nN(v)-\nN(\tv) \right)  + (\phi_{m}(\|v\|_{\xX_{2}(0,s)}) - \phi_{m}(\|\tv\|_{\xX_{2}(0,s)})) \nN(\widetilde{v})
\end{equation}
and estimate separately with observation $|(\phi_{m}(\|v\|_{\chi_{2}(0,s)})-\phi_{m}(\|\tv\|_{\xX_{2}(0,s)}))|\lesssim \|w\|_{\xX(0,a_{l+1})}$.
\item Without loss of generality, we only consider $ \phi_{m}(\|v\|_{\xX_{2}(0,s)}) = 0$ for $s \geq a_{l}$ and $\|\widetilde{v}\|_{\xX_{2}(a_{l},a_{l+1})} \leq \eta$. We observe
\begin{equation}
\phi_{m}(\|v\|_{\xX_{2}(0,s)})\nN(v)-\phi_{m}(\|\tv\|_{\xX_{2}(0,s)}) \nN(\tv) = \phi_{m}(\|v\|_{\xX_{2}(0,s)})-\phi_{m}(\|\tv\|_{\xX_{2}(0,s)}) \nN(\tv)\;,
\end{equation}
and observe again $|(\phi_{m}(\|v\|_{\chi_{2}(0,s)})-\phi_{m}(\|\tv\|_{\xX_{2}(0,s)}))|\lesssim \|w\|_{\xX(0,a_{l+1})}$. The bound for this situation also follows. 
\end{itemize}

Now, in every interval $[a_{l}, a_{l+1}]$, using Strichartz and triangle inequalities, we have
\begin{equation}
\|w\|_{\xX(a_{l},a_{l+1})} \lesssim \|w(a_{l})\|_{L_{x}^{2}} + \eta^{4} \|w\|_{\xX_{2}(0, a_{l+1})} + \sum_{i} \|S_{i}\|_{\xX(a_{l},a_{l+1})}\;. 
\end{equation}
Absorbing $\eta^{4}\|w\|_{\xX_{2}(0, a_{l+1})}$ into $\|w\|_{\xX(0,a_{l+1})}$, and applying the estimates in Lemma~\ref{lem:esfornewe}, we get
\begin{equation}
\begin{aligned}
\|w\|_{\xX(0,a_{l+1})} \lesssim &\|w\|_{\xX(0,a_{l})} + \sum_{i} \|M_{i}^{*}(t)\|_{L_{t}^{5}(c,T)} + \|\tilde{M}_{i}^{*}|\\
+&\sum_{i=1,2} \|g_{i}^{*}(t)\|_{L_{t}^{5}(c,T)} + \tilde{g}_{i}^{*}  + \|\pi^{n}\|^{\theta} + C_{m} \sup_j \|\Delta_{j}W\|_{L_{x}^{\infty}};. 
\end{aligned}
\end{equation}
Iterating the above bound $\sim m^{5}/\eta^{5}+1\leq C_{m}$ times, we get
\begin{equation}
\begin{aligned}
\|w\|_{\xX(0,T)} \leq &C_{m} \Big\{\|w\|_{\xX[0,c]} + \sum_{i}\|M_{i}^{*}(t)\|_{L_{t}^{5}(c,T)} + \|\tilde{M}_{i}^{*}|\\
+&\sum _{i=1,2} \|g_{i}^{*}(t)\|_{L_{t}^{5}(c,T)} + \tilde{g}_{i}^{*} +\|\pi^{n}\|^{\theta} + C_{m} \sup_j \|\Delta_{j}W\|_{L_{x}^{\infty}}\Big\}\;. 
\end{aligned}
\end{equation}
Taking $L^{\rho}_{\omega}$-norm on both sides, we obtain
\begin{equation}
\|w\|_{L^{\rho}_{\omega}\xX(0,T)} \leq C_{m} \Big(\|w\|_{\xX[0,c]} + h^{\theta}\|w\|_{L^{\rho}_{\omega}} + \|\pi^{(n)}\|^{\theta} \Big)\;. 
\end{equation}
We now choose $h$ sufficiently small depending on $C_{m}$, and $\|\pi^{(n)}\|$ small enough  depending on $C_{m}$ and $\delta_{0}$. The desired bound then follows.

\appendix

\section{Some useful inequalities}
We will need following regularity estimate for the linear Sch\"odinger estimate, it is natural in the sense the natural scaling of Schr\"odinger will trade two space derivative into one time derivative.

\begin{lem} \label{le:semi_group_cont}
	For every $0<\alpha<1$, we have
	\begin{equation}
	\|e^{i t \Delta} \psi - \psi\|_{L_{x}^{2}(\R^d)} \lesssim_{\alpha,d} t^{\frac{\alpha}{2}} \|\psi\|_{H^{\alpha}(\R^d)}, 
	\end{equation}
	uniformly over $t \in (0,1)$. 
\end{lem}
\begin{proof}
Applying Plancherel Theorem, we have
\begin{equation}
\|e^{it\Delta}\psi-\psi\|_{L_{x}^{2}}=\|(1-e^{it\|\xi|^{2}})\hat{\psi}(\xi)\|_{L_{\xi}^{2}}.
\end{equation}
Split the $\R^{d}_{\xi}$ into the $|\xi|\leq \frac{1}{t^{1/2}}$ and $\xi\geq \frac{1}{t^{1/2}}$, use $|e^{it|\xi|^{2}}-1|\lesssim t|\xi|^{2}$ in the first region, and $|e^{it|\xi|^{2}}-1|\leq 2$ in the second one, then the desired estimate follows.
\end{proof}
This has an immediate consequence which will be useful in our analysis, we summarize it as the following lemma

\begin{lem} \label{le:op_change_order}
Let $V$ be some nice Schwartz function, let $\tau_{1}\leq \tau_{2}\leq \tau_{3}$, $\tau_{3}-\tau_{1} \leq 1$, then we have
\begin{equation}
\|e^{i(\tau_{3}-\tau_{2})\Delta} \big(V e^{i(\tau_{2}-\tau_{1})\Delta} \psi\big)-\psi\|_{L_{x}^{2}}\lesssim (\tau_{3}-\tau_{1})^{\alpha/2}\|\psi\|_{H^{\alpha}}
\end{equation}
The bounds depend on $V$ but are uniform in $\psi$ and the $\tau_i$'s. 
\end{lem}

\endappendix

\bibliographystyle{Martin}
\bibliography{Refs}

\end{document}